%% file: Main.tex
\documentclass[12pt,oneside]{amsart}
\usepackage{graphicx}
\usepackage{amsmath}
\usepackage{amssymb}
\usepackage{amsthm}
\usepackage{esint}
\usepackage[margin=1.1in]{geometry}
\usepackage{xcolor}
\usepackage[square,numbers,sort&compress]{natbib}
\usepackage{hyperref}
\usepackage{mathtools}

\newtheorem{theorem}{Theorem}[section]
\newtheorem{lem}[theorem]{Lemma}
\newtheorem*{theorem*}{Theorem}
\theoremstyle{plain}
\newtheorem{cor}[theorem]{Corollary}
\newtheorem{prop}[theorem]{Proposition}
\newtheorem{conjecture}[theorem]{Conjecture}

\theoremstyle{remark}
\newtheorem{remark}[theorem]{Remark}

\numberwithin{equation}{section}

\begin{document}

\title[fractional Yamabe flow]{Convergence of the fractional Yamabe flow
for arbitrary initial energy}

\author{Jingeon An}

\address{Department of Mathematics and Computer Science, University of Basel, Spiegelgasse 1, 4052 Basel, Switzerland}

\email{jingeon.an@icloud.com}

\author{Hardy Chan}

\address{Department of Mathematics and Computer Science, University of Basel, Spiegelgasse 1, 4052 Basel, Switzerland}

\email{hardy.chan@unibas.ch}

\author{Pak Tung Ho}
\address{Department of Mathematics, Tamkang University, Tamsui, New Taipei City 251301, Taiwan}

\email{paktungho@yahoo.com.hk}

\subjclass[2020]{Primary  35K55; Secondary  35B40}

\begin{abstract}
Since the seminal paper of Graham and Zworski (Invent. Math. 2003), conformal geometric problems are studied in the fractional setting. We consider the convergence of fractional Yamabe flow, which is previously known under small initial energy assumption. Inspired by the deep work of Brendle (J. Diff. Geom. 2005), we obtain the full convergence result for arbitrary initial energy, whenever the (fractional) positive mass conjecture is valid. 
    
\end{abstract}

\maketitle

\input{Introduction}
\input{Preliminary}
\input{Blow-up_analysis}

\input{The_case_u_infty=0}

\input{The_case_u_infty_gtr_0}

\input{Proof_of_main_prop}

\input{Estimates_on_bubbles}
\input{Appendix}

\section*{Acknowledgement}

We thank Professor Simon Brendle for insightful discussions.
The first author also expresses gratitude to Professor Martin Mayer for the fruitful discussion and his suggestion on his paper, which ultimately led to the resolution of the problem.
The first and second authors have received funding from the Swiss National Science Foundation
under Grant PZ00P2\_202012/1. 
The third author was supported by   the National Science and Technology Council (NSTC),
Taiwan, with grant Number: 112-2115-M-032 -006 -MY2. The second author is grateful for the warm hospitality he received when visiting the Department of Applied Mathematics and Data Science, Tamkang University, where part of the work was completed.

\medskip

\bibliographystyle{abbrvnat}
\bibliography{Bib}

\end{document}

%% file: Introduction.tex
\section{Introduction}

As one of the landmarks in geometric analysis, the problem of finding a metric in a given conformal class of a closed manifold with constant scalar curvature is known as the Yamabe problem (see \cite{aubin1998some,schoen1984conformal,trudinger1968remarks,yamabe1960deformation}). Later in his seminal paper, Hamilton \cite{hamilton1988ricci} introduced the parabolic version of the Yamabe problem, now called Yamabe flow: For any compact Riemannian manifold $(M,g_0)$ of dimension $n\geq 2$, consider the following evolution equation for the metric $g(t)$ such that
\begin{equation}\label{eq: classical yamabe flow in curvature terms}
    \begin{cases}
        \partial_t g&=-(R^g-s^g)g\\
        g(0)&=g_0,
    \end{cases}
\end{equation}
where $R^g$ denotes the scalar curvature of $g$, and $s^g$ is the averaged scalar curvature, i.e.
\[
    s^g:=\frac{1}{\mu_{g}(M)}\int_M R^g\,d\mu_{g}.
\]
It is possible to rewrite \eqref{eq: classical yamabe flow in curvature terms} as a parabolic PDE in the conformal factor, under the conformal change
\[
    g=u^{\frac{4}{n-2}}g_0.
\]
Then, in terms of the conformal Laplacian $P^{g_0}$, \eqref{eq: classical yamabe flow in curvature terms} is equivalent to a fast diffusion equation
\[
    \frac{n-2}{n+2}\partial_t \left(u^{\frac{n+2}{n-2}}\right)=-P^{g_0}u+s^gu^{\frac{n+2}{n-2}}.
\]
In this formulation, the Yamabe flow resulted in an extensive amount of literature in geometric analysis, see, e.g. \cite{Brendle,brendle2007convergence,chow1992yamabe,schwetlick2003convergence,ye1994global}.

In a seminal paper \cite{graham2003scattering}, Graham and Zworski introduced a one-parameter family of conformal fractional Laplacians $P_\gamma^g$, of order $2\gamma\in (0,n)$, defined on the conformal infinity of a Poincaré--Einstein manifold. Remarkably, these operators coincide with the GJMS operators \cite{graham1992conformally} for integer values of $\gamma$. This also includes the fractional Laplacian as a special case, when the ambient manifold is the hyperbolic space. When $\gamma\in(0,1)$, an extension problem is formulated by Chang and Gonz\'alez \cite{chang2011fractional}, analogous to the one posed by Caffarelli--Silvestre \cite{Caffarelli-Silvestre} in the Euclidean setting.

Correspondingly, one defines an interpolated curvature quantity $R_\gamma^g:=P_\gamma^g(1)$ for each $\gamma\in (0,n/2)$, which is just scalar curvature for $\gamma=1$, and the $Q$-curvature for $\gamma=2$. This interpolated notion of curvature is studied in \cite{chang2011fractional,del2012singular,gonzalez2013fractional,jin2014fractional,qing2006compactness,mayer2021asymptotics,mayer2024fractional}.

Then, one naturally considers the fractional Yamabe flow, namely  \eqref{eq: classical yamabe flow in curvature terms} but equipped with this generalized notion of curvature. 
Recently, the convergence of the fractional Yamabe flow is one of the main interests and studied in \cite{jin2014fractional,daskalopoulos2017weak,ChanSireSun}. As we expect from the classical theory, it is known by Chan, Sire, and Sun \cite{ChanSireSun} that the fractional Yamabe flow converges strongly, but the argument in \cite{ChanSireSun} proves the convergence only under a small energy assumption on the initial metric $g_0$. The aim of this paper is to remove this initial energy assumption and prove the strong convergence for arbitrary initial energy, as in the classical case treated by Brendle \cite{Brendle}.

The main difficulty of generalizing Brendle's strategy \cite{Brendle} to the fractional case is, of course, nonlocality, because of which one cannot derive certain important estimates from the pointwise product rule as in the local case $\gamma=1$. One needs to resort to the manifold in extension where, unfortunately, two versions of Schoen's bubble (i.e. Aubin--Talenti's bubble glued to Green's function) arise. More precisely, delicate estimates are required to compare the ``glue-and-extend'' and ``extend-and-glue'' versions. In this direction, Mayer and Ndiaye \cite{mayer2021asymptotics,mayer2024fractional} recently made an important contribution by deriving asymptotics of Poisson's kernel for fractional conformal Laplacians and used them to prove the existence of the fractional Yamabe problem. Building on their work, we prove a novel and crucial pointwise gradient estimate (in order to obtain \eqref{eq: L 2n / n+2gamma estimate on bubble}). Collecting all the ingredients leads to the final resolution of the expected convergence theory of the fractional Yamabe flow. 


\vspace{0.5cm}

We now introduce the flow under study. Suppose $M$ is the conformal infinity of a
Poincar\'{e}-Einstein manifold $(X,g_+)$. Moreover, we let $\gamma\in (0,1)\subset (0,n/2)$, and define the fractional Yamabe functional
\begin{equation}\label{eq: fractional Yamabe functional def}
    E(u):=\frac{\int_M uP_\gamma^{g_0}u\,d\mu_{g_0}}{\left(\int_M u^{\frac{2n}{n-2\gamma}}d\mu_{g_0}\right)^{\frac{n-2\gamma}{n}}},
\end{equation}
and fractional Yamabe energy 
\begin{equation}\label{eq: fractional yamabe energy def}
    Y_\gamma(M,[g_0])=\inf_{0\neq u\in H^\gamma(M)}E(u),
\end{equation}
where $P_\gamma^{g_0}$ is the conformal fractional Laplacian (for the precise definition, see Section \ref{sec: pre} and \cite[Section 2]{ChanSireSun}). Recall that the conformal fractional Laplacian satisfies
\[
    P_\gamma^{g_0}(uf)=u^{\frac{n+2\gamma}{n-2\gamma}}P_\gamma^g(f),\quad\forall f\in C^\infty(M),
\]
under the conformal change 
\begin{equation}\label{eq: conformal change}
    g=u^{\frac{4}{n-2\gamma}}g_0.
\end{equation}
In particular, on usual Euclidean metric $(\mathbb{R}^n,|dx|^2)$, we have $P_\gamma^{|dx|^2}=(-\Delta_{\mathbb{R}^n})^\gamma$.

The volume element on $(M,g_0)$ is denoted by $d\mu_{g_0}$. By replacing $g_0$ by its constant multiple we may assume the $(M,g_0)$ has unit volume, that is to say, $\mu_{g_0}(M)=1$. With a conformal metric \eqref{eq: conformal change}, we write
\[
    d\mu_g=u^{\frac{2n}{n-2\gamma}}d\mu_{g_0}.
\]

Assume $Y_\gamma(M,[g_0])>0$ and $\lambda_1(g_+)\geq\frac{n^2}{2}-\gamma^2$.
Then for each $x\in M$, there exists a unique positive Green's function $G(\cdot,x)$
on $\overline{X}\setminus\{x\}$, for the conformal fractional Laplacian $P_\gamma^g$ (see \cite[Proposition 1.5]{KimMussoWei}). Then we have the following Positive Mass conjecture, stated for fractional case: 

\begin{conjecture}[Positive Mass Conjecture]
Assume that $\gamma\in (0,1)$, $n>2\gamma$
Suppose that $Y_\gamma(M,[g_0])>0$. Fix any $y\in M$. Then there exists a small neighborhood of $y$ in $(\overline{X},\bar{g})$, which is diffeomorphic to a small neighborhood $\mathcal{N}\subset\mathbb{R}^{n+1}_+$ of $0$, such that 
\[
G(x,0)=c_{n,\gamma}|x|^{-n+2\gamma}+A_0+\psi(x)~~\mbox{ for }x\in\mathcal{N},
\]
Here, $c_{n,\gamma}=\pi^{-n/2}2^{-2\gamma}\Gamma(\gamma)^{-1}\Gamma(\frac{n}{2}-\gamma)$
and $\psi$ is a function in $\mathcal{N}$ satisfying
$$|\psi(x)|\leq C|x|^{\min\{1,2\gamma\}}~~\mbox{ and }~~|\nabla\psi(x)|\leq C|x|^{\min\{0,2\gamma-1\}}$$
for some constant $C>0$. Moreover, $A_0>0$, unless $(X,\bar{g})$ is conformally diffeomorphic to the standard unit ball. 
\end{conjecture}
The Positive Mass Conjecture for the fractional operators $P_\gamma^g$ is out of reach at the moment, due to the non-locality and sheer lack of tools to treat this case. So in the following we naturally assume that Positive Mass holds.

Let $R_\gamma^g:=P_\gamma^g(1)=u^{-\frac{n+2\gamma}{n-2\gamma}}P_\gamma^{g_0}(u)$ be the fractional curvature. As previously mentioned, this is the scalar curvature when $\gamma=1$ and $Q$-curvature when $\gamma=2$. Its average is denoted by
\[
s_\gamma^g=\int_MR_\gamma^g d\mu_g.
\]

The \textit{fractional Yamabe flow} is the evolution of the metric $g=g(t)$
given by
\begin{equation}\label{1.1}
\begin{cases}
    \frac{n-2\gamma}{4}\partial_tg&=(s_\gamma^g-R_\gamma^g)g\\
    g(0)&=g_0.
\end{cases}
\end{equation}
Writing with conformal change $g=u^{\frac{4}{n-2\gamma}}g_0$, the fractional Yamabe flow (\ref{1.1}) is equivalent to
\begin{equation}\label{1.2}
\begin{cases}
    \frac{\partial u}{\partial t}&=(s_\gamma^g-R_\gamma^g)u\\
    u(0)&\equiv 1.
\end{cases}
\end{equation}

Chan, Sire, and Sun \cite{ChanSireSun} proved the following:
\begin{theorem*}[Theorem 1.2 in \cite{ChanSireSun}]
For $\gamma\in (0,1)$, assume that $Y_\gamma(M,[g_0])>0$,
$\lambda_1(g_+)\geq\frac{n^2}{2}-\gamma^2$ and, in the case
$\gamma\in (\frac{1}{2},1)$, $H=0$, where $H$ denotes the mean curvature of $\partial_\infty X=M$.
Furthermore, assume the Positive Mass Conjecture holds with $A_0>0$.
If the initial energy of the initial metric $g(0)$ satisfies
\begin{equation}\label{eq: energy bound condition}
    s_\gamma^{g(0)}\leq \left[Y_\gamma(M,[g_0])^{\frac{n}{2\gamma}}+Y_\gamma(\mathbb{S}^n)^{\frac{n}{2\gamma}}\right]^{\frac{2\gamma}{n}},
\end{equation}
the fractional Yamabe flow \eqref{1.1} converges.
\end{theorem*}

In this paper, we improve the above Theorem by removing the initial energy bound assumption. In particular, we have the following:

\begin{theorem}\label{main_thm}
Let $n\geq 2$. For $\gamma\in (0,1)$, assume that $Y_\gamma(M,[g_0])>0$,
$\lambda_1(g_+)\geq\frac{n^2}{2}-\gamma^2$ and, in the case
$\gamma\in (\frac{1}{2},1)$, $H=0$, where $H$ denotes the mean curvature of $\partial_\infty X=M$.
Furthermore, assume the Positive Mass Conjecture holds with $A_0>0$.
Then the fractional Yamabe flow \eqref{1.1} converges.
\end{theorem}

\begin{remark}
    \begin{enumerate}
        \item In our main theorem, we did not specify in which sense the flow converges. Following previous works, the flow is globally defined and H\"older continuous. However, we expect that the flow is actually smooth, from the parabolic regularity theory, and the convergence can be stated in the $C^\infty(M)$ sense.
        \item We state the main theorem with $n\geq 2$, following previous works \cite{mayer2021asymptotics,mayer2024fractional}. However, we expect the same result is true for $n=1$, and $\gamma\leq 1/2$. 
    \end{enumerate}
\end{remark}

\subsection*{Notations}

We write here a list of symbols used throughout the paper. 

\begin{center}
\begin{tabular}{p{2cm}p{13cm}}
$\gamma$ & Real number in $(0,1)$\\
$(X,g_+)$ & Poincaré-Einstein manifold\\
$(M,g_0)$ & Conformal infinity of $(X,g_+)$\\
$\mu_g$ & Volume form on $(M,g)$\\
$E$ & Fractional Yamabe functional, see \eqref{eq: fractional yamabe energy def}\\
$Y_\gamma(M,[g_0])$ & Fractional Yamabe energy, see \eqref{eq: fractional yamabe energy def}\\
$F$ & Fractional Yamaabe type functional, see \eqref{48}\\
$B^g_r(x)$ & Geodesic ball in $(M,g)$ of radius $r>0$, centered at $x\in (M,g)$\\
$B_r(x)$ & Abbreviation of $B^{g_0}_r(x)$\\
$B_r^{g_+,+}(x)$ & Geodesic ball in $(X,g_+)$ of radius $r>0$, centered at $x\in \overline{X}$\\
$u$ & Conformal factor\\
$R_\gamma^g$ & Fractional scalar curvature of order $\gamma$, defined as $P_\gamma^g(1)=u^{-\frac{n+2\gamma}{n-2\gamma}}P_\gamma^{g_0}(u)$\\
$s_\gamma^g$ & Averaged fractional scalar curvature, $\int_M R^g_\gamma\,d\mu_g$\\
$u_{(x,\epsilon)}$ & Test bubbles, see Appendix \ref{Appendix B}\\
$o(1)$ & quantity that converges to $0$ in $\nu\rightarrow+\infty$
\end{tabular}
\end{center}

%% file: Preliminary.tex
\section{Preliminaries}\label{sec: pre}

In this section, following \cite{chang2011fractional,ChanSireSun}, we first provide the definition of fractional conformal Laplacian. Then, we give some preliminary results on the Yamabe flow and sketch the proof of Theorem \ref{main_thm}, provided Proposition \ref{prop1.1}. In the following, we will always assume that $(M,[g_0])$ is the conformal infinity of $(X,g_+)$, both equipped with appropriate metrics. A function $y$ is a defining function on $M$ in $X$ if 
\[
    y=0\text{ on } M,\quad y>0\text{ in }X\quad\text{and}\quad |dy|\neq 0\text{ on }M.
\]
For each representative metric $g\in [g_0]$, there is a unique geodesic defining function $y$ associated to $g$ such that $g_+=y^{-2}(dy^2+g_y)$ where $g_y$ is one parameter family of metrics on $M$ satisfying $g_y|_M=g$. Moreover, $g_y$ has an asymptotic expansion which contains only even powers of $y$, at least up to degree $n$. Consequently $M$ is totally geodesic in $(X,y^2 g_+)$. 

Assume $\lambda_1(-\Delta_{g_+})>\frac{n^2}{4}-\gamma^2$. According to \cite[Theorem 4.7]{chang2011fractional}, there exists a special defining function $y^*$ enjoying the following Caffarelli-Silvestre type extension property. For any smooth function $u$ on $M$, one can find extension $U$ on $X$ satisfying 
\begin{equation}
    \left\{
        \begin{alignedat}{3}
            -\text{div}((y^*)^{1-2\gamma}\nabla U)&=0&&\quad\text{in}\quad (X,(y^*)^2g_+)\\
            U&=u&&\quad\text{on}\quad (M,g)\\
            (P_\gamma^g-R^g_\gamma)u&=-c_\gamma\lim_{y^*\rightarrow 0}(y^*)^{1-2\gamma}\partial_{y^*}U&&\quad\text{on}\quad(M,g),
        \end{alignedat}
    \right.
\end{equation}
where $c_\gamma$ is a positive constant which can be found in \cite{chang2011fractional}.

The flow (\ref{1.1}) preserves the volume of $(M,g)$, i.e.
$$\frac{d}{dt}\left(\int_Md\mu_g\right)=0~~\mbox{ for all }t\geq 0.$$
Hence, if we assume that $\displaystyle\int_Md\mu_{g_0}=1$, we have
\begin{equation}\label{1.5}
\int_Mu^{\frac{2n}{n-2\gamma}}d\mu_{g_0}=\int_Md\mu_{g}=1~~\mbox{ for all }t\geq 0.
\end{equation}
Along the flow (\ref{1.1}), there holds (c.f \cite[Lemma 2.2]{ChanSireSun})
\begin{equation}\label{1.4}
\frac{d}{dt}s_\gamma^{g}=-2\int_M|R_\gamma^g-s_\gamma^g|^2d\mu_g.
\end{equation}
It follows from (\ref{1.5}) and the definition of $Y_\gamma(M,[g_0])$ that
$s_\gamma^{g}\geq Y_\gamma(M,[g_0])>0$. This together with (\ref{1.4}) implies that
$s_\infty=\displaystyle\lim_{t\to\infty}s_\gamma^g$ exists.

As the key lemma of the proof, we have the following:

\begin{prop}\label{prop1.1}
Let  $\{t_\nu:\nu\in\mathbb{N}\}$ be a sequence of times such that $t_\nu\to\infty$ as $\nu\to\infty$.
Then there exist $\delta\in (0,1)$ and a constant $C$ such that, after passing to a subsequence, we have
$$s_\gamma^{g(t_\nu)}-s_\infty\leq C\left(\int_M|R_\gamma^{g(t_\nu)}-s_\gamma^{g(t_\nu)}|^{\frac{2n}{n+2\gamma}}d\mu_{g(t_\nu)}\right)^{\frac{n+2\gamma}{2n}(1+\delta)}$$
for all integers $\nu$ in that sequence. Note that $\delta$ and $C$ may depend on the
sequence   $\{t_\nu:\nu\in\mathbb{N}\}$.
\end{prop}

If Proposition \ref{prop1.1} is proved, then
one can follow the argument in \cite{ChanSireSun}
(see Propositions 3.6-3.10 in \cite{ChanSireSun})
to prove that the function $u=u(t)$ in (\ref{1.2}) is uniformly bounded, i.e.
\begin{equation}\label{1.3}
C^{-1}\leq\inf_Mu(t)\leq \sup_M u(t)\leq C
\end{equation}
for some positive constant $C$ independent of $t$.
With the uniform bound in (\ref{1.3}), one can then follow the proof of
\cite[Proposition 3.1]{ChanSireSun} to conclude that
$u$ is H\"{o}lder continuous, i.e. for any fixed $\frac{n}{2\gamma}<p<\frac{n+2\gamma}{2\gamma}$, $\alpha=2\gamma-\frac{n}{p}>0$, and $T>0$, there exists $C(T)>0$ such that
\[
|u(x_1,t_1)-u(x_2,t_2)|\leq C(T)\big(|t_1-t_2|^{\frac{\alpha}{2}}+d(x_1,x_2)^\alpha\big),\quad\forall t_1,t_2\in[0,T],\quad\forall x_1,x_2\in M.
\]
From this, following standard parabolic theory, one can derive higher order estimates, and then it concludes that the flow (\ref{1.1}) converges. This proves Theorem \ref{main_thm}.

Therefore, it reduces to prove Proposition \ref{prop1.1}.

%% file: Blow-up_analysis.tex
\section{Blow-up analysis}

Let  $\{t_\nu:\nu\in\mathbb{N}\}$ be a sequence of times
such that $t_\nu\to\infty$ as $\nu\to\infty$.
For abbreviation, let $u_\nu=u(t_\nu)$ and
$g_\nu=g(t_\nu)=u(t_\nu)^{\frac{4}{n-2\gamma}}g_0=u_\nu^{\frac{4}{n-2\gamma}}g_0$.
The normalization condition (\ref{1.5}) implies that
\begin{equation}\label{2.1}
\int_Md\mu_{g_\nu}=\int_Mu_\nu^{\frac{2n}{n-2\gamma}}d\mu_{g_0}=1~~\mbox{ for all }\nu\in\mathbb{N}.
\end{equation}
Moreover, it follows from \cite[Lemma 3.2]{ChanSireSun} that, if we define
\[
    F_p(g):=\int_M |R_\gamma^g-s_\gamma^g|^p\,d\mu_g,
\]
then for $1\leq p<\frac{n+2\gamma}{2\gamma}$, we have
\begin{equation*}
\lim_{t\to\infty} F_{p}(g(t))=0.
\end{equation*}
Choosing $p=\frac{2n}{n+2\gamma}$ and $t=t_{\nu}$ yields
\begin{equation*}
\int_M|R_\gamma^{g_\nu}-s_\infty|^{\frac{2n}{n+2\gamma}}d\mu_{g_\nu}\to 0~~\mbox{ as }\nu\to\infty,
\end{equation*}
which gives
\begin{equation}\label{2.3}
\int_M\Big|P_\gamma^{g_0}u_\nu-s_{\infty}u_\nu^{\frac{n+2\gamma}{n-2\gamma}}\Big|^{\frac{2n}{n+2\gamma}}d\mu_{g_0}\to 0~~\mbox{ as }\nu\to\infty.
\end{equation}

\begin{lem}\label{lem2.1}
Let $\{u_\nu:\nu\in\mathbb{N}\}$ be a sequence of positive functions satisfying
\eqref{2.1} and \eqref{2.3}. After passing to a subsequence if necessary, we can find
a nonnegative integer $m$, a nonnegative smooth function $u_\infty$ and a
sequence of $m$-tuplets $(x_{k,\nu}^*,\varepsilon_{k,\nu}^*)_{1\leq k\leq m}$
with the following properties:\\
(i) The function $u_\infty$ satisfies the equation
\begin{equation}\label{2.4}
P_\gamma^{g_0}u_\infty=s_\infty u_\infty^{\frac{n+2\gamma}{n-2\gamma}}~~\mbox{ in }M.
\end{equation}
(ii) $\varepsilon_1^*$ converges to zero for any $i$, and for all $i\neq j$, we have
\begin{equation}\label{2.5}
\frac{\varepsilon_1^*}{\varepsilon_2^*}+\frac{\varepsilon_2^*}{\varepsilon_1^*}
+\frac{d(x_1^*,x_2^*)}{\varepsilon_1^*\varepsilon_2^*}\to\infty
\end{equation}
as $\nu\to\infty$.\\
(iii) We have
\begin{equation}\label{2.6}
u_\nu-u_\infty-\sum_{k=1}^mu_{(x_{k,\nu}^*,\varepsilon_{k,\nu}^*)}\to 0~~\mbox{ in }H^\gamma(M),
\end{equation}
as $\nu\to\infty$. Here the function $u_{(x_{k,\nu}^*,\varepsilon_{k,\nu}^*)}$
are fractional test bubbles constructed in \cite[Section 3.2]{mayer2024fractional}, up to some multiplicative constant (see Appendix \ref{Appendix B} and \eqref{eq: def of test bubble}). 
\end{lem}

\begin{remark}\label{rmk: positivity of bubbles}
    Schoen's bubble is constructed by attaching Green's function $G$ and standard fractional bubble profile
    \[
    \left(\frac{\varepsilon}{\varepsilon^2+|x|^2}\right)^{\frac{n-2\gamma}{2}},
    \]
    by cut-off. Therefore from the positivity of $G$, Schoen's bubble is strictly positive everywhere on $M$. 
\end{remark}

\begin{proof}
\eqref{2.3} controls $(H^\gamma(M))^*=H^{-\gamma}(M)$ norm of the derivative of the Yamabe functional. Boundedness of the Yamabe functional follows from the normalization and monotonicity \eqref{1.4}. Thus such sequence satisfies the Palais-Smale condition. As pointed out in \cite{ChanSireSun}, the conclusion follows from \cite{PalatucciPisante}, after repeating same argument on the manifold setting (also note that Schoen's bubble converges to standard bubble in $H^\gamma(M)$). See also \cite[Theorem 1.3]{FangG}. 
\end{proof}

From \cite[Lemma 3.5]{ChanSireSun} (see also \cite{DQ}), we have
either $u_\infty>0$ or $u_\infty\equiv 0$. We are going to deal with these two cases separately.

Because of the interaction of bubbles decay (see Proposition \ref{prop: estimation of interaction of bubbles}), we have
\begin{equation*}
\begin{split}
1 &=\lim_{\nu\to\infty}\int_Mu_\nu^{\frac{2n}{n-2\gamma}}d\mu_{g_0}\\
&=\lim_{\nu\to\infty}\left(\int_Mu_\infty^{\frac{2n}{n-2\gamma}}d\mu_{g_0}
+\sum_{k=1}^m\int_Mu_{(x_{k,\nu}^*,\varepsilon_{k,\nu}^*)}^{\frac{2n}{n-2\gamma}}d\mu_{g_0}\right)\\
&=\left(\frac{E(u_\infty)}{s_\infty}\right)^{\frac{n}{2\gamma}}+m\left(\frac{Y_\gamma(\mathbb{S}^n)}{s_\infty}\right)^{\frac{n}{2\gamma}},
\end{split}
\end{equation*}

which gives

\begin{equation}\label{2.7}
s_\infty=\left[E(u_\infty)^{\frac{n}{2\gamma}}+mY_\gamma(\mathbb{S}^n)^{\frac{n}{2\gamma}}\right]^{\frac{2\gamma}{n}}.
\end{equation}

This will be an essential part in the proof of Proposition \ref{prop1.1}.

We conclude this section by listing crucial estimates for Schoen’s bubbles.

\begin{prop}\label{prop: estimation of interaction of bubbles}
    Suppose $\varepsilon_2> \varepsilon_1$. We have
    \begin{equation}\label{eq: fractional conformal Lapalcian interaction term estimate}
        \int_M u_{(x_1,\varepsilon_1)}P_\gamma^{g_0}u_{(x_2,\varepsilon_2)}\leq C\left(\frac{\varepsilon_2^2+d(x_1,x_2)^2}{\varepsilon_1\varepsilon_2}\right)^{-\frac{n-2\gamma}{2}},
    \end{equation}
    \begin{equation}\label{eq: nonlinearity interaction term estimate}
        \int_M u_{(x_1,\varepsilon_1)}^{\frac{n+2\gamma}{n-2\gamma}}u_{(x_2,\varepsilon_2)}\,d\mu_{g_0}\leq C\left(\frac{\varepsilon_2^2+d(x_1,x_2)^2}{\varepsilon_1\varepsilon_2}\right)^{-\frac{n-2\gamma}{2}},
    \end{equation}
    \begin{equation}\label{eq: Main interaction of bubble term estimate}
        \int_M u_{(x_1,\varepsilon_1)}\left(P_\gamma^{g_0}u_{(x_2,\varepsilon_2)}-F(u_{(x_2,\varepsilon_2)})u_{(x_2,\varepsilon_2)}^{\frac{n+2\gamma}{n-2\gamma}}\right)\,d\mu_{g_0}\leq o_{\varepsilon_2}(1)\left(\frac{\varepsilon_2^2+d(x_1,x_2)^2}{\varepsilon_1\varepsilon_2}\right)^{-\frac{n-2\gamma}{2}},
    \end{equation}
    and
    \begin{equation}\label{eq: L 2n / n+2gamma estimate on bubble}
        \int_M\left| P_\gamma^{g_0}u_{(x,\varepsilon)}-s_\infty u_{(x,\varepsilon)}^{\frac{n+2\gamma}{n-2\gamma}} \right|^{\frac{2n}{n+2\gamma}}\,d\mu_{g_0}= o_\varepsilon(1).
    \end{equation}
    Moreover, Schoen's bubble satisfies energy estimate
    \begin{equation}\label{eq: energy for Schoen's bubble estimation}
        E(u_{(x,\varepsilon)})\leq Y_\gamma(\mathbb{S}^n),
    \end{equation}
    for $\varepsilon$ small enough.
\end{prop}
\begin{proof}
    The interaction estimates \eqref{eq: fractional conformal Lapalcian interaction term estimate} and \eqref{eq: nonlinearity interaction term estimate} follows from \cite[Lemma 3.5]{mayer2024fractional}. \eqref{eq: Main interaction of bubble term estimate} is from \cite[Lemma 3.5]{mayer2024fractional} and \cite[Lemma 4.2]{mayer2024fractional}. \eqref{eq: energy for Schoen's bubble estimation} is taken from the selfaction estimate \cite[Lemma 4.2]{mayer2024fractional}, with \cite[(18)]{mayer2024fractional}. Finally, \eqref{eq: L 2n / n+2gamma estimate on bubble} follows from maximum principle argument on the extension domain (for the detail, see Appendix \ref{appendix C}).
\end{proof}

%% file: The_case_u_infty=0.tex
\section{The case \texorpdfstring{$u_\infty\equiv 0$}{u\_infty=0}}

In this section, we always assume $u_\infty\equiv 0$. For every $\nu\in\mathbb{N}$, we denote by $\mathcal{A}_\nu$ the set of all $m$-tuplets $(x_k,\varepsilon_k,\alpha_k)_{1\leq k\leq m}\in (M\times\mathbb{R}_+,\mathbb{R}_+)^m$
such that
$$d(x_k,x_{k,\nu}^*)\leq \varepsilon_{k,\nu}^*,~~\frac{1}{2}\leq\frac{\varepsilon_k}{\varepsilon_{k,\nu}^*}\leq 2,~~
\frac{1}{2}\leq\alpha_k\leq 2$$
for all $1\leq k\leq m$. Moreover, we can find an $m$-tuplet
$(x_{k,\nu},\varepsilon_{k,\nu},\alpha_{k,\nu})_{1\leq k\leq m}\in\mathcal{A}_\nu$ such that
\begin{equation}\label{3.1}
\begin{split}
&\int_M\Big(u_\nu-\sum_{k=1}^m\alpha_{k,\nu}u_{(x_{k,\nu},\varepsilon_{k,\nu})}\Big)P_\gamma^{g_0}\Big(u_\nu-\sum_{k=1}^m\alpha_{k,\nu}u_{(x_{k,\nu},\varepsilon_{k,\nu})}\Big)d\mu_{g_0}\\
&\leq \int_M\Big(u_\nu-\sum_{k=1}^m\alpha_{k}u_{(x_{k},\varepsilon_{k})}\Big)P_\gamma^{g_0}\Big(u_\nu-\sum_{k=1}^m\alpha_{k}u_{(x_{k},\varepsilon_{k})}\Big)d\mu_{g_0}
\end{split}
\end{equation}
for all $(x_k,\varepsilon_k,\alpha_k)_{1\leq k\leq m}\in\mathcal{A}_\nu$.

\begin{prop}\label{prop3.1}
(i) For all $i\neq j$, we have
$$\frac{\varepsilon_{i,\nu}}{\varepsilon_{j,\nu}}+\frac{\varepsilon_{j,\nu}}{\varepsilon_{i,\nu}}
+\frac{d(x_{i,\nu},x_{j,\nu})}{\varepsilon_{i,\nu}\varepsilon_{j,\nu}}\to\infty
$$
as $\nu\to\infty$.\\
(ii) We have
$$u_\nu-\sum_{k=1}^m\alpha_{k,\nu}u_{(x_{k,\nu},\varepsilon_{k,\nu})}\to 0~~\mbox{ in }H^\gamma(M)$$
as $\nu\to\infty$.
\end{prop}
\begin{proof}
Part (i) is identical to the proof of \cite[Proposition 5.1(i)]{Brendle}.
For Part (ii), it follows from definition of $(x_{k,\nu},\varepsilon_{k,\nu},\alpha_{k,\nu})_{1\leq k\leq m}$
that
\begin{equation*}
\begin{split}
&\int_M\Big(u_\nu-\sum_{k=1}^m\alpha_{k,\nu}u_{(x_{k,\nu},\varepsilon_{k,\nu})}\Big)P_\gamma^{g_0}\Big(u_\nu-\sum_{k=1}^m\alpha_{k,\nu}u_{(x_{k,\nu},\varepsilon_{k,\nu})}\Big)d\mu_{g_0}\\
&\leq \int_M\Big(u_\nu-\sum_{k=1}^m u_{(x_{k,\nu}^*,\varepsilon_{k,\nu}^*)}\Big)P_\gamma^{g_0}\Big(u_\nu-\sum_{k=1}^m u_{(x_{k,\nu}^*,\varepsilon_{k,\nu}^*)}\Big)d\mu_{g_0}.
\end{split}
\end{equation*}
By Lemma \ref{lem2.1}, the expression on the right hand side tends to $0$
as $\nu\to\infty$. This proves the assertion.
\end{proof}

\begin{prop}\label{prop3.2}
We have
$$d(x_{k,\nu},x_{k,\nu}^*)=o(1)\varepsilon_{k,\nu}^*,~~\frac{\varepsilon_k}{\varepsilon_{k,\nu}^*}=1+o(1),~~
\alpha_k=1+o(1)$$
for all $1\leq k\leq m$. In particular,
$(x_{k,\nu},\varepsilon_{k,\nu},\alpha_{k,\nu})_{1\leq k\leq m}$
is an interior point of $\mathcal{A}_\nu$ if $\nu$ is sufficiently large.
\end{prop}
\begin{proof}
We find
\begin{equation*}
\begin{split}
&\Bigg\|\sum_{k=1}^m\alpha_{k,\nu}u_{(x_{k,\nu},\varepsilon_{k,\nu})}-\sum_{k=1}^m u_{(x_{k,\nu}^*,\varepsilon_{k,\nu}^*)}\Bigg\|_{H^\gamma(M)}\\
&\leq \Bigg\|u_\nu-\sum_{k=1}^m\alpha_{k,\nu}u_{(x_{k,\nu},\varepsilon_{k,\nu})}\Bigg\|_{H^\gamma(M)}
+\Bigg\|u_\nu-\sum_{k=1}^m u_{(x_{k,\nu}^*,\varepsilon_{k,\nu}^*)}\Bigg\|_{H^\gamma(M)}=o(1)
\end{split}
\end{equation*}
by Lemma \ref{lem2.1} and Proposition \ref{prop3.1}.
From this, the assertion follows.
\end{proof}

In the sequel, by rearrangement, we assume that
$$\varepsilon_{i,\nu}\leq \varepsilon_{j,\nu}\quad\mbox{for}\quad i\leq j.$$
We now decompose the function $u_\nu$ as
$$u_\nu=v_\nu+w_\nu,$$
where
$$v_\nu=\sum_{k=1}^m\alpha_{k,\nu}u_{(x_{k,\nu},\varepsilon_{k,\nu})},$$
and
$$w_\nu=u_\nu-\sum_{k=1}^m\alpha_{k,\nu}u_{(x_{k,\nu},\varepsilon_{k,\nu})}.$$
It follows from Proposition \ref{prop3.1} that 
$$\int_Mw_\nu P_\gamma^{g_0} w_\nu d\mu_{g_0}=o(1).$$

\begin{prop}\label{prop: some estimates on u^{n+2gamma/n-2gamma}w_nu}
(i) For every $1\leq k\leq m$, we have
$$\left|\int_Mu_{(x_{k,\nu},\varepsilon_{k,\nu})}^{\frac{n+2\gamma}{n-2\gamma}}w_\nu d\mu_{g_0}\right|\leq o(1)\|w_\nu\|_{H^\gamma(M)}.$$
(ii) For every $1\leq k\leq m$, we have
$$\left|\int_Mu_{(x_{k,\nu},\varepsilon_{k,\nu})}^{\frac{n+2\gamma}{n-2\gamma}}\frac{\varepsilon_{k,\nu}^2-d(x_{k,\nu},x)^2}{\varepsilon_{k,\nu}^2+d(x_{k,\nu},x)^2}w_\nu d\mu_{g_0}\right|\leq o(1)\|w_\nu\|_{H^\gamma(M)}.$$
(iii) For all $1\leq k\leq m$,
$$\left|\int_Mu_{(x_{k,\nu},\varepsilon_{k,\nu})}^{\frac{n+2\gamma}{n-2\gamma}}\frac{\varepsilon_{k,\nu}\exp^{-1}_{x_{k,\nu}}(x)}{\varepsilon_{k,\nu}^2+d(x_{k,\nu},x)^2}w_\nu d\mu_{g_0}\right|\leq o(1)\|w_\nu\|_{H^\gamma(M)}.$$
\end{prop}
\begin{proof}
By definition of $(x_{k,\nu},\varepsilon_{k,\nu},\alpha_{k,\nu})_{1\leq k\leq m}$, by taking derivative with respect to $\alpha_k$, we have 
$$\int_Mw_\nu P_\gamma^{g_0}u_{(x_{k,\nu},\varepsilon_{k,\nu})} d\mu_{g_0}=0,$$
for all $1\leq k\leq m$. From \eqref{eq: L 2n / n+2gamma estimate on bubble}
$$\Big\|P_\gamma^{g_0}u_{(x_{k,\nu},\varepsilon_{k,\nu})}-s_\infty u_{(x_{k,\nu},\varepsilon_{k,\nu})}^{\frac{n+2\gamma}{n-2\gamma}}\Big\|_{L^{\frac{2n}{n+2\gamma}}(M)}=o(1),$$
we conclude that
\begin{equation*}
\begin{split}
&\left|s_\infty \int_Mu_{(x_{k,\nu},\varepsilon_{k,\nu})}^{\frac{n+2\gamma}{n-2\gamma}}w_\nu d\mu_{g_0}\right|\\
&=\left|\int_M\Big(P_\gamma^{g_0}u_{(x_{k,\nu},\varepsilon_{k,\nu})}-s_\infty u_{(x_{k,\nu},\varepsilon_{k,\nu})}^{\frac{n+2\gamma}{n-2\gamma}}\Big)w_\nu d\mu_{g_0}\right|\\
&\leq \Big\|P_\gamma^{g_0}u_{(x_{k,\nu},\varepsilon_{k,\nu})}-s_\infty u_{(x_{k,\nu},\varepsilon_{k,\nu})}^{\frac{n+2\gamma}{n-2\gamma}}\Big\|_{L^{\frac{2n}{n+2\gamma}}(M)}
\|w_\nu\|_{L^{\frac{2n}{n-2\gamma}}(M)}\\
&\leq o(1)\|w_\nu\|_{H^\gamma(M)}
\end{split}
\end{equation*}
for all $1\leq k\leq m$. This proves (i).

The remaining statements follow similarly, taking derivative with respect to $\varepsilon_{k,\nu}$ and $x_{k,\nu}$, for (ii) and (iii), respectively. 
\end{proof}

\begin{prop}\label{prop: estimate on u^{4gamma/n-2gamma}w_nu^2}
If $\nu$ is sufficiently large, then we have 
\begin{equation*}
\frac{n+2\gamma}{n-2\gamma}s_\infty\int_M\sum_{j=1}^m  u_{(x_{j,\nu},\varepsilon_{j,\nu})}^{\frac{4\gamma}{n-2\gamma}}w_\nu^2 d\mu_{g_0}
\leq (1-c)\int_M w_\nu P_\gamma^{g_0} w_\nu d\mu_{g_0}
\end{equation*}
for some positive constant $c$ independent of $\nu$. 
\end{prop}
\begin{proof}
Suppose this is not true. Upon rescaling (as in \cite[Lemma 5.2]{ChanSireSun}), we obtain a sequence of functions $\{\widetilde{w}_\nu:\nu\in\mathbb{N}\}$ such that
\begin{equation}\label{64}
\int_M\widetilde{w}_\nu P_\gamma^{g_0} \widetilde{w}_\nu d\mu_{g_0}=1
\end{equation}
and
\begin{equation}\label{65}
\lim_{\nu\to\infty}\frac{n+2\gamma}{n-2\gamma}s_\infty\int_M\sum_{j=1}^m  u_{(x_{j,\nu},\varepsilon_{j,\nu})}^{\frac{4\gamma}{n-2\gamma}}\widetilde{w}_\nu^2 d\mu_{g_0}
\geq 1.
\end{equation}
From (\ref{64}), we have
\begin{equation}\label{66}
\int_M|\widetilde{w}_\nu|^{\frac{2n}{n-2\gamma}} d\mu_{g_0}
\leq Y_\gamma(M,[g_0])^{-\frac{n}{n-2\gamma}}.
\end{equation}
In view of Lemma \ref{lem2.1}, denote
\[
    N_\nu^2:=\min\left\{\frac{\varepsilon_{2,\nu}}{\varepsilon_{1,\nu}}, \dots, \frac{\varepsilon_{m,\nu}}{\varepsilon_{m-1,\nu}}\right\}.
\]
Then $N_\nu\to\infty$, $N_\nu\varepsilon_{j,\nu}\to 0$ for all $1\leq j\leq m$, and 
\begin{equation}\label{67}
\frac{1}{N_\nu}\frac{\varepsilon_{j,\nu}+d(x_{i,\nu},x_{j,\nu})}{\varepsilon_{i,\nu}}\to\infty
\end{equation}
for all $i<j$. 
$$\Omega_{j,\nu}=B_{N_\nu\varepsilon_{j,\nu}}(x_{j,\nu})\setminus\bigcup_{i=1}^{j-1} B_{N_\nu\varepsilon_{i,\nu}}(x_{i,\nu})$$
for every $1\leq j\leq m$. In view of (\ref{64}) and (\ref{65}), we can find an integer $1\leq j\leq m$ such that
\begin{equation}\label{69}
\lim_{\nu\to\infty}\int_M u_{(x_{j,\nu},\varepsilon_{j,\nu})}^{\frac{4\gamma}{n-2\gamma}}\widetilde{w}_\nu^2 d\mu_{g_0}>0
\end{equation}
and
\begin{equation}\label{70}
\lim_{\nu\to\infty}\int_{\Omega_{j,\nu}}\widetilde{w}_\nu P_\gamma^{g_0} \widetilde{w}_\nu d\mu_{g_0}
\leq \lim_{\nu\to\infty}\frac{n+2\gamma}{n-2\gamma}s_\infty\int_M u_{(x_{j,\nu},\varepsilon_{j,\nu})}^{\frac{4\gamma}{n-2\gamma}}\widetilde{w}_\nu^2 d\mu_{g_0}.
\end{equation}
We now define a sequence of functions $\hat{w}_\nu: T_{x_{j,\nu}}M\to\mathbb{R}$ by
$$\hat{w}_\nu(\xi)=\varepsilon_{j,\nu}^{\frac{n-2\gamma}{2}}\widetilde{w}_\nu(\exp_{x_{j,\nu}}(\varepsilon_{j,\nu}\xi))$$
for $\xi\in T_{x_{j,\nu}}M$.
The sequence $\{\hat{w}_\nu:\nu\in\mathbb{N}\}$ satisfies 
$$\lim_{\nu\to\infty}\int_{\{\xi\in T_{x_{j,\nu}}M: |\xi|\leq N_\nu\}}\hat{w}_\nu(\xi)(-\Delta_{\mathbb{R}^n})^\gamma \hat{w}_\nu(\xi)d\xi
\leq 1$$
and
$$\lim_{\nu\to\infty}\int_{\{\xi\in T_{x_{j,\nu}}M: |\xi|\leq N_\nu\}}|\hat{w}_\nu(\xi)|^{\frac{2n}{n-2\gamma}}d\xi\leq Y_\gamma(M,[g_0])^{-\frac{n}{n-2\gamma}}.$$
Hence, if we take the weak limit as $\nu\to\infty$
such that $\hat{w}_\nu \rightharpoonup \hat{w}$ in $H_\textup{loc}^\gamma(\mathbb{R}^n)$,
we then obtain a function $\hat{w}:\mathbb{R}^n\to\mathbb{R}$ such that
\begin{equation}\label{3.8}
\int_{\mathbb{R}^n}\frac{\hat{w}(\xi)^2}{(1+|\xi|^2)^{2\gamma}}d\xi>0
\end{equation}
and 
\begin{equation}\label{3.9}
\int_{\mathbb{R}^n}\hat{w}(\xi) (-\Delta_{\mathbb{R}^n})^\gamma\hat{w}(\xi)d\xi\leq\alpha_{n,\gamma}^{\frac{4\gamma}{n-2\gamma}}\frac{n+2\gamma}{n-2\gamma}
\int_{\mathbb{R}^n}\frac{\hat{w}(\xi)^2}{(1+|\xi|^2)^{2\gamma}}d\xi
\end{equation}
where $\alpha_{n,\gamma}$ is a constant given by
\begin{equation}\label{3.11}
\alpha_{n,\gamma}=2^{\frac{n-2\gamma}{2}}\left(\frac{\Gamma(\frac{n}{2}+\gamma)}{\Gamma(\frac{n}{2}-\gamma)}\right)^{\frac{n-2\gamma}{4\gamma}}.
\end{equation}
Moreover, it follows from Proposition \ref{prop: some estimates on u^{n+2gamma/n-2gamma}w_nu} that
\begin{equation}\label{3.10}
\begin{split}
\int_{\mathbb{R}^n}\left(\frac{1}{1+|\xi|^2}\right)^{\frac{n+2\gamma}{2}}\hat{w}(\xi)d\xi&=0,\\
\int_{\mathbb{R}^n}\left(\frac{1}{1+|\xi|^2}\right)^{\frac{n+2\gamma}{2}}\frac{1-|\xi|^2}{1+|\xi|^2}\hat{w}(\xi)d\xi&=0,\\
\int_{\mathbb{R}^n}\left(\frac{1}{1+|\xi|^2}\right)^{\frac{n+2\gamma}{2}}\frac{\xi}{1+|\xi|^2}\hat{w}(\xi)d\xi&=0.
\end{split}
\end{equation}
Since $\hat{w}$ satisfies (\ref{3.9}) and (\ref{3.10}),
we can follow the argument in the proof of \cite[Lemma 5.1]{ChanSireSun}
to conclude that $\hat{w}\equiv 0$, which contradicts to (\ref{3.8}).
This proves the proposition.
\end{proof}

\begin{cor}\label{cor3.5}
 If $\nu$ is sufficiently large, then we have
\begin{equation*}
\frac{n+2\gamma}{n-2\gamma}s_\infty\int_Mv_\nu^{\frac{4\gamma}{n-2\gamma}}w_\nu^2 d\mu_{g_0}
\leq (1-c)\int_M w_\nu P_\gamma^{g_0} w_\nu d\mu_{g_0}
\end{equation*}
for some positive constant $c$ independent of $\nu$.
\end{cor}
\begin{proof}
By the definition of $v_\nu$, we have
$$\int_M\Big|v_\nu^{\frac{4\gamma}{n-2\gamma}}-\sum_{j=1}^m u_{(x_{k,\nu},\varepsilon_{k,\nu})}^{\frac{4\gamma}{n-2\gamma}}\Big|^{\frac{n}{2\gamma}}d\mu_{g_0}=o(1).$$
Hence, the assertion follows from Proposition \ref{prop: estimate on u^{4gamma/n-2gamma}w_nu^2}. 
\end{proof}

In the following, we denote
\begin{equation}\label{48}
F(v)=\frac{\int_MvP_\gamma^{g_0}v d\mu_{g_0}}{\int_M v^{\frac{2n}{n-2\gamma}}d\mu_{g_0}}.
\end{equation}

\begin{prop}\label{prop5.6}
There holds 
$$E(v_\nu)\leq \Big(\sum_{k=1}^mE(u_{(x_{k,\nu},\varepsilon_{k,\nu})})^{\frac{n}{2\gamma}}\Big)^{\frac{2\gamma}{n}}$$
if $\nu$ is sufficiently large.
\end{prop}
\begin{proof}
Using the identity
\begin{equation*}
\begin{split}
\int_Mv_\nu P_\gamma^{g_0}v_\nu d\mu_{g_0}
&=\int_M\sum_{k=1}^m\alpha_{k,\nu}^2u_{(x_{k,\nu},\varepsilon_{k,\nu})} P_\gamma^{g_0}(u_{(x_{k,\nu},\varepsilon_{k,\nu})})  d\mu_{g_0}\\
&\hspace{4mm}+2\int_M\sum_{i<j}\alpha_{i,\nu}\alpha_{j,\nu}u_{(x_{i,\nu},\varepsilon_{i,\nu})} P_\gamma^{g_0}(u_{(x_{j,\nu},\varepsilon_{j,\nu})})  d\mu_{g_0},
\end{split}
\end{equation*}
we obtain
\begin{equation*}
\begin{split}
E(v_\nu)\left(\int_Mv_\nu^{\frac{2n}{n-2\gamma}}d\mu_{g_0}\right)^{\frac{n-2\gamma}{n}}
&=\int_M\sum_{k=1}^m\alpha_{k,\nu}^2F(u_{(x_{k,\nu},\varepsilon_{k,\nu})}) u_{(x_{k,\nu},\varepsilon_{k,\nu})}^{\frac{2n}{n-2\gamma}} d\mu_{g_0}\\
&\hspace{4mm}+2\int_M\sum_{i<j}^m\alpha_{i,\nu}\alpha_{j,\nu}u_{(x_{i,\nu},\varepsilon_{i,\nu})} P_\gamma^{g_0}(u_{(x_{j,\nu},\varepsilon_{j,\nu})})  d\mu_{g_0}.
\end{split}
\end{equation*}
Moreover, using the positivity of Schoen's bubble (see Remark \ref{rmk: positivity of bubbles}), we have 
\begin{equation*}
\begin{split}
&\Bigg(\sum_{k=1}^mE(u_{(x_{k,\nu},\varepsilon_{k,\nu})})^{\frac{n}{2\gamma}}\Bigg)^{\frac{2\gamma}{n}}\left(\int_Mv_\nu^{\frac{2n}{n-2\gamma}}d\mu_{g_0}\right)^{\frac{n-2\gamma}{n}}\\
&=\Bigg(\int_M\Big(\sum_{k=1}^mF(u_{(x_{k,\nu},\varepsilon_{k,\nu})})^{\frac{n}{2\gamma}} u_{(x_{k,\nu},\varepsilon_{k,\nu})}^{\frac{2n}{n-2\gamma}}\Big) d\mu_{g_0}\Bigg)^{\frac{2\gamma}{n}}\left(\int_Mv_\nu^{\frac{2n}{n-2\gamma}}d\mu_{g_0}\right)^{\frac{n-2\gamma}{n}}\\
&\geq \int_M\Big(\sum_{k=1}^mF(u_{(x_{k,\nu},\varepsilon_{k,\nu})})^{\frac{n}{2\gamma}} u_{(x_{k,\nu},\varepsilon_{k,\nu})}^{\frac{2n}{n-2\gamma}}\Big)^{\frac{2\gamma}{n}}v_\nu^2d\mu_{g_0}\\
&\geq
\int_M\sum_{k=1}^m\alpha_{k,\nu}^2F(u_{(x_{k,\nu},\varepsilon_{k,\nu})}) u_{(x_{k,\nu},\varepsilon_{k,\nu})}^{\frac{2n}{n-2\gamma}} d\mu_{g_0}\\
&\hspace{4mm}+2\int_M\sum_{i<j}\alpha_{i,\nu}\alpha_{j,\nu}\Big(F(u_{(x_{i,\nu},\varepsilon_{i,\nu})})^{\frac{n}{2\gamma}}u_{(x_{i,\nu},\varepsilon_{i,\nu})}^{\frac{2n}{n-2\gamma}}+F(u_{(x_{j,\nu},\varepsilon_{j,\nu})})^{\frac{n}{2\gamma}}u_{(x_{j,\nu},\varepsilon_{j,\nu})}^{\frac{2n}{n-2\gamma}}\Big)^{\frac{2\gamma}{n}}\\
&\hspace{24mm}\times u_{(x_{i,\nu},\varepsilon_{i,\nu})}
u_{(x_{j,\nu},\varepsilon_{j,\nu})} d\mu_{g_0},
\end{split}
\end{equation*}
by H\"{o}lder's inequality. Consider a pair $i<j$.
We can find positive constants $c$ and $C$ independent of $\nu$ such that
\begin{equation*}
u_{(x_{i,\nu},\varepsilon_{i,\nu})}^{\frac{n+2\gamma}{n-2\gamma}}u_{(x_{j,\nu},\varepsilon_{j,\nu})}\geq
c\Big(\frac{\varepsilon_{j,\nu}^2+d(x_{i,\nu},x_{j,\nu})^2}{\varepsilon_{i,\nu}\varepsilon_{j,\nu}}\Big)^{-\frac{n-2\gamma}{2}}\varepsilon_{i,\nu}^{-n}
\end{equation*}
and
\begin{equation*}
u_{(x_{i,\nu},\varepsilon_{i,\nu})}u_{(x_{j,\nu},\varepsilon_{j,\nu})}^{\frac{n+2\gamma}{n-2\gamma}}\leq
C\Big(\frac{\varepsilon_{j,\nu}^2+d(x_{i,\nu},x_{j,\nu})^2}{\varepsilon_{i,\nu}\varepsilon_{j,\nu}}\Big)^{-\frac{n+2\gamma}{2}}\varepsilon_{i,\nu}^{-n}
\end{equation*}
if $d(x_{i,\nu},x)\leq \varepsilon_{i,\nu}$ and $\nu$ is sufficiently large. From this,
it follows that
\begin{equation}\label{4.13}
\begin{split}
&\Big(F(u_{(x_{i,\nu},\varepsilon_{i,\nu})})^{\frac{n}{2\gamma}}u_{(x_{i,\nu},\varepsilon_{i,\nu})}^{\frac{2n}{n-2\gamma}}
+F(u_{(x_{j,\nu},\varepsilon_{j,\nu})})^{\frac{n}{2\gamma}}u_{(x_{j,\nu},\varepsilon_{j,\nu})}^{\frac{2n}{n-2\gamma}}\Big)^{\frac{2\gamma}{n}}u_{(x_{i,\nu},\varepsilon_{i,\nu})}
u_{(x_{j,\nu},\varepsilon_{j,\nu})}\\
&\geq F(u_{(x_{j,\nu},\varepsilon_{j,\nu})})u_{(x_{i,\nu},\varepsilon_{i,\nu})}u_{(x_{j,\nu},\varepsilon_{j,\nu})}^{\frac{n+2\gamma}{n-2\gamma}}
+c\Big(\frac{\varepsilon_{j,\nu}^2+d(x_{i,\nu},x_{j,\nu})^2}{\varepsilon_{i,\nu}\varepsilon_{j,\nu}}\Big)^{-\frac{n-2\gamma}{2}}\varepsilon_{i,\nu}^{-n}
1_{\{d(x_{i,\nu},x)\leq \varepsilon_{i,\nu}\}}
\end{split}
\end{equation}
for $\nu$ sufficiently large. Integration over $M$ yields
\begin{equation*}
\begin{split}
&\Bigg(\sum_{k=1}^mE(u_{(x_{k,\nu},\varepsilon_{k,\nu})})^{\frac{n}{2\gamma}}\Bigg)^{\frac{2\gamma}{n}}\left(\int_Mv_\nu^{\frac{2n}{n-2\gamma}}d\mu_{g_0}\right)^{\frac{n-2\gamma}{n}}\\
&\geq \int_M\sum_{k=1}^m\alpha_{k,\nu}^2F(u_{(x_{k,\nu},\varepsilon_{k,\nu})}) u_{(x_{k,\nu},\varepsilon_{k,\nu})}^{\frac{2n}{n-2\gamma}} d\mu_{g_0}\\
&\hspace{4mm}+2\int_M\sum_{i<j}\alpha_{i,\nu}\alpha_{j,\nu}F(u_{(x_{j,\nu},\varepsilon_{j,\nu})})u_{(x_{i,\nu},\varepsilon_{i,\nu})}u_{(x_{j,\nu},\varepsilon_{j,\nu})}^{\frac{n+2\gamma}{n-2\gamma}}d\mu_{g_0}\\
&\hspace{4mm}+c\Big(\frac{\varepsilon_{j,\nu}^2+d(x_{i,\nu},x_{j,\nu})^2}{\varepsilon_{i,\nu}\varepsilon_{j,\nu}}\Big)^{-\frac{n-2\gamma}{2}}.
\end{split}
\end{equation*}
Putting these facts together, we get
\begin{equation*}
\begin{split}
&E(v_\nu)\left(\int_Mv_\nu^{\frac{2n}{n-2\gamma}}d\mu_{g_0}\right)^{\frac{n-2\gamma}{n}}\\
&\hspace{4mm}\leq \Bigg(\sum_{k=1}^mE(u_{(x_{k,\nu},\varepsilon_{k,\nu})})^{\frac{n}{2\gamma}}\Bigg)^{\frac{2\gamma}{n}}\left(\int_Mv_\nu^{\frac{2n}{n-2\gamma}}d\mu_{g_0}\right)^{\frac{n-2\gamma}{n}}\\
&\hspace{8mm}+2\int_M\sum_{i<j}\alpha_{i,\nu}\alpha_{j,\nu}u_{(x_{i,\nu},\varepsilon_{i,\nu})}\Big(P_\gamma^{g_0}u_{(x_{j,\nu},\varepsilon_{j,\nu})}-F(u_{(x_{j,\nu},\varepsilon_{j,\nu})})u_{(x_{j,\nu},\varepsilon_{j,\nu})}^{\frac{n+2\gamma}{n-2\gamma}}\Big)d\mu_{g_0}\\
&\hspace{8mm}-c\Big(\frac{\varepsilon_{j,\nu}^2+d(x_{i,\nu},x_{j,\nu})^2}{\varepsilon_{i,\nu}\varepsilon_{j,\nu}}\Big)^{-\frac{n-2\gamma}{2}}.
\end{split}
\end{equation*}
Finally, from \eqref{eq: Main interaction of bubble term estimate} the assertion follows.
\end{proof}

Therefore, we conclude this section with the next corollary, which will be used in the proof of Proposition \ref{prop1.1}.

\begin{cor}\label{cor5.7}
If $\nu$ is sufficiently large, then
$$E(v_\nu)\leq \Big(mY_\gamma(\mathbb{S}^n)^{\frac{n}{2\gamma}}\Big)^{\frac{2\gamma}{n}}.$$
\end{cor}

%% file: The_case_u_infty_gtr_0.tex
\section{The case \texorpdfstring{$u_\infty>0$}{u\_infty>0}}

We have the following proposition; its proof can be found in \cite{ChanSireSun}.

\begin{prop}[Proposition 4.1 in \cite{ChanSireSun}]\label{prop6.1}
There exists a sequence of smooth functions $\{\psi_a: a\in\mathbb{N}\}$ and a sequence of positive real numbers
$\{\lambda_a:a\in\mathbb{N}\}$ with the following properties:\\
(i) For all $a\in\mathbb{N}$,
$$P_\gamma^{g_0}\psi_a=\lambda_a u_\infty^{\frac{4\gamma}{n-2\gamma}}\psi_a~~\mbox{ in }M.$$
(ii) For all $a,b\in\mathbb{N}$,
\begin{equation}\label{eq: def of inner product and orthogonality}
    (\psi_a,\psi_b):=\int_M\psi_a\psi_b u_\infty^{\frac{4\gamma}{n-2\gamma}}d\mu_{g_0}=
    \left\{
      \begin{array}{ll}
        1, & \hbox{if $a=b$;} \\
        0, & \hbox{if $a\neq b$.}
      \end{array}
    \right.
\end{equation}
(iii) The span of $\{\psi_a: a\in\mathbb{N}\}$ is dense in $L^2(M)$.\\
(iv) $\lambda_a\to\infty$ as $a\to\infty$.
\end{prop}

Let $A$ be a finite subset of $\mathbb{N}$ such that $\lambda_a>\frac{n+2\gamma}{n-2\gamma}s_\infty$
for all $a\not\in A$. We denote by $\Pi$ the projection operator
\begin{equation*}
\Pi(f)=\sum_{a\not\in A}\left(\int_M\psi_a f d\mu_{g_0}\right)\psi_a u_\infty^{\frac{4\gamma}{n-2\gamma}}
=f-\sum_{a\in A}\left(\int_M\psi_a f d\mu_{g_0}\right)\psi_a u_\infty^{\frac{4\gamma}{n-2\gamma}}.
\end{equation*}

\begin{remark}
    Note that this definition facilitates the computations for Lemma \ref{lem: def of perturbed uinfty} below and is not the canonical projection with respect to the inner product defined in \eqref{eq: def of inner product and orthogonality}, which would read
    \[
        \tilde{\Pi}(f)=\sum_{a\notin A}\left(\int_M\psi_a f u_\infty^{\frac{4\gamma}{n-2\gamma}}\,d\mu_{g_0}\right)\psi_a.
    \]
\end{remark}

\begin{lem}\label{lem6.2}
For every $1\leq p<\infty$, we can find a constant $C$ such that for any $f\in \dot{H}^\gamma(M)$,
\begin{equation*}
\|f\|_{L^p(M)}\leq C\Big\|P_\gamma^{g_0} f-\frac{n+2\gamma}{n-2\gamma}s_\infty u_\infty^{\frac{4\gamma}{n-2\gamma}} f\Big\|_{L^p(M)}+C\sup_{a\in A}\left|\int_M u_\infty^{\frac{4\gamma}{n-2\gamma}}\psi_a fd\mu_{g_0}\right|.
\end{equation*}
\end{lem}
\begin{proof}
Suppose that it is not true. By compactness, we can find a function $f\in L^p(M)$ satisfying
$\|f\|_{L^p(M)}=1$,
\begin{equation*}
\int_M u_\infty^{\frac{4\gamma}{n-2\gamma}}\psi_a fd\mu_{g_0}=0
\end{equation*}
for all $a\in A$ and
\begin{equation*}
P_\gamma^{g_0} f-\frac{n+2\gamma}{n-2\gamma}s_\infty u_\infty^{\frac{4\gamma}{n-2\gamma}} f=0
\end{equation*}
in the sense of distributions.
Hence, if we use the function $\psi_a$ as a test function, then we obtain
$$\left(\lambda_a-\frac{n+2\gamma}{n-2\gamma}s_\infty\right)\int_Mu_\infty^{\frac{4\gamma}{n-2\gamma}}\psi_a fd\mu_{g_0}=0$$
for all $a\in\mathbb{N}$. In particular, we have
\begin{equation*}
\int_M u_\infty^{\frac{4\gamma}{n-2\gamma}}\psi_a fd\mu_{g_0}=0
\end{equation*}
for all $a\not\in A$.
Thus, from the density and the assumption $u_\infty>0$, we conclude that $f=0$. This is a contradiction.
\end{proof}

\begin{lem}\label{lem6.3}
(i) There exists a constant $C$ such that for any $f\in \dot{H}^\gamma(M)$,
\begin{equation*}
\begin{split}
&\|f\|_{L^{\frac{n+2\gamma}{n-2\gamma}}(M)}\\
&\leq C\left\|\Pi\left(P_\gamma^{g_0} f-\frac{n+2\gamma}{n-2\gamma}s_\infty u_\infty^{\frac{4\gamma}{n-2\gamma}} f\right)\right\|_{L^{\frac{n(n+2\gamma)}{n^2+4\gamma^2}}(M)}+C\sup_{a\in A}\left|\int_M u_\infty^{\frac{4\gamma}{n-2\gamma}}\psi_a fd\mu_{g_0}\right|.
\end{split}
\end{equation*}
(ii) There exists a constant $C$ such that
\begin{equation*}
\|f\|_{L^{1}(M)}\leq C\left\|\Pi\left(P_\gamma^{g_0} f-\frac{n+2\gamma}{n-2\gamma}s_\infty u_\infty^{\frac{4\gamma}{n-2\gamma}} f\right)\right\|_{L^1(M)}+C\sup_{a\in A}\left|\int_M u_\infty^{\frac{4\gamma}{n-2\gamma}}\psi_a fd\mu_{g_0}\right|.
\end{equation*}
\end{lem}
\begin{proof}
(i)
It follows from the fractional Sobolev inequality and elliptic estimate that (c.f. \cite[Theorem 6.5]{DPV})
\begin{equation*}
\begin{split}
&\|f\|_{L^{\frac{n+2\gamma}{n-2\gamma}}(M)}\\
&\leq C\left\|P_\gamma^{g_0} f-\frac{n+2\gamma}{n-2\gamma}s_\infty u_\infty^{\frac{4\gamma}{n-2\gamma}} f\right\|_{L^{\frac{n(n+2\gamma)}{n^2+4\gamma^2}}(M)}
+C\|f\|_{L^{\frac{n(n+2\gamma)}{n^2+4\gamma^2}}(M)}.
\end{split}
\end{equation*}
Using Lemma \ref{lem6.2}, we obtain
\begin{equation*}
\begin{split}
&\|f\|_{L^{\frac{n+2\gamma}{n-2\gamma}}(M)}\\
&\leq C\left\|P_\gamma^{g_0} f-\frac{n+2\gamma}{n-2\gamma}s_\infty u_\infty^{\frac{4\gamma}{n-2\gamma}} f\right\|_{L^{\frac{n(n+2\gamma)}{n^2+4\gamma^2}}(M)}
+C\sup_{a\in A}\left|\int_M u_\infty^{\frac{4\gamma}{n-2\gamma}}\psi_a fd\mu_{g_0}\right|.
\end{split}
\end{equation*}
By definition of $\Pi$, we have
\begin{equation}\label{92}
\begin{split}
&P_\gamma^{g_0} f-\frac{n+2\gamma}{n-2\gamma}s_\infty u_\infty^{\frac{4\gamma}{n-2\gamma}} f\\
&\hspace{4mm}=\Pi\Big(P_\gamma^{g_0} f-\frac{n+2\gamma}{n-2\gamma}s_\infty u_\infty^{\frac{4\gamma}{n-2\gamma}} f\Big)\\
&\hspace{8mm}+\sum_{a\in A}\Big(\lambda_a-\frac{n+2\gamma}{n-2\gamma}s_\infty\Big)\left(\int_Mu_\infty^{\frac{4\gamma}{n-2\gamma}} \psi_a f d\mu_{g_0}\right) u_\infty^{\frac{4\gamma}{n-2\gamma}}\psi_a.
\end{split}
\end{equation}
This implies
\begin{equation}\label{93}
\begin{split}
&\left\|P_\gamma^{g_0} f-\frac{n+2\gamma}{n-2\gamma}s_\infty u_\infty^{\frac{4\gamma}{n-2\gamma}} f\right\|_{L^q(M)}\\
&\leq\left\|\Pi\Big(P_\gamma^{g_0} f-\frac{n+2\gamma}{n-2\gamma}s_\infty u_\infty^{\frac{4\gamma}{n-2\gamma}} f\Big)\right\|_{L^q(M)}
+C\sup_{a\in A}\left|\int_Mu_\infty^{\frac{4\gamma}{n-2\gamma}} \psi_a f d\mu_{g_0}\right|.
\end{split}
\end{equation}
Putting these facts together, the assertion follows.\\
(ii) It follows from Lemma \ref{lem6.2} that
\begin{equation*}
\|f\|_{L^1(M)}\leq C\left\|P_\gamma^{g_0} f-\frac{n+2\gamma}{n-2\gamma}s_\infty u_\infty^{\frac{4\gamma}{n-2\gamma}} f\right\|_{L^1(M)}
+C\sup_{a\in A}\left|\int_M u_\infty^{\frac{4\gamma}{n-2\gamma}}\psi_a fd\mu_{g_0}\right|.
\end{equation*}
Combining this with
(\ref{92}) and (\ref{93}) yields
\begin{equation*}
\|f\|_{L^{1}(M)}\leq C\left\|\Pi\left(P_\gamma^{g_0} f-\frac{n+2\gamma}{n-2\gamma}s_\infty u_\infty^{\frac{4\gamma}{n-2\gamma}} f\right)\right\|_{L^1(M)}+C\sup_{a\in A}\left|\int_M u_\infty^{\frac{4\gamma}{n-2\gamma}}\psi_a fd\mu_{g_0}\right|.
\end{equation*}
This proves the assertion.
\end{proof}

We are going to construct functions $\bar{u}_z$, which are perturbations of $u_\infty$ in a finite dimensional subspace, and whose derivatives satisfy nice orthogonality conditions. We have the following two lemmas in \cite{ChanSireSun}.

\begin{lem}[Lemma 4.3 in \cite{ChanSireSun}]\label{lem: def of perturbed uinfty}
There exists a positive real number $\zeta$ with the following signature: for every vector $z\in \mathbb{R}^{|A|}$
with $|z|\leq \zeta$, there exists a smooth positive function $\bar{u}_z$ such that
\begin{equation*}
\int_Mu_\infty^{\frac{4\gamma}{n-2\gamma}}(\bar{u}_z-u_\infty)\psi_a d\mu_{g_0}=z_a~~\mbox{ for all }a\in A,
\end{equation*}
and
$$\Pi\left(P_\gamma^{g_0}\bar{u}_z-s_\infty \bar{u}_z^{\frac{n+2\gamma}{n-2\gamma}}\right)=0.$$
Furthermore, the map $z\mapsto\bar{u}_z$ is real analytic.
\end{lem}

\begin{lem}[Lemma 4.4 in \cite{ChanSireSun}]\label{lem6.5}
There exists a real number $0<\delta<1$ such that
$$E(\bar{u}_z)-E(u_\infty)
\leq C\sup_{a\in A}\left|\int_M\psi_a\left(P_\gamma^{g_0}\bar{u}_z-s_\infty \bar{u}_z^{\frac{n+2\gamma}{n-2\gamma}}\right)d\mu_{g_0}\right|^{1+\delta}$$
if $|z|$ is sufficiently small.
\end{lem}

For every $\nu\in\mathbb{N}$, we denote by $\mathcal{A}_\nu$ the set of all pairs
 $\big(z,(x_k,\varepsilon_k,\alpha_k)_{1\leq k\leq m}\big)\in \mathbb{R}^{|A|}\times (M\times\mathbb{R}_+,\mathbb{R}_+)^m$
such that
$$|z|\leq\zeta$$
and
$$d(x_k,x_{k,\nu}^*)\leq \varepsilon_{k,\nu}^*,~~\frac{1}{2}\leq\frac{\varepsilon_k}{\varepsilon_{k,\nu}^*}\leq 2,~~
\frac{1}{2}\leq\alpha_k\leq 2$$
for all $1\leq k\leq m$. Moreover, we can find a pair
$\big(z_\nu,(x_{k,\nu},\varepsilon_{k,\nu},\alpha_{k,\nu})_{1\leq k\leq m}\big)\in\mathcal{A}_\nu$ such that
\begin{equation}\label{108}
\begin{split}
&\int_M\Big(u_\nu-\bar{u}_{z_\nu}-\sum_{k=1}^m\alpha_{k,\nu}u_{(x_{k,\nu},\varepsilon_{k,\nu})}\Big)P_\gamma^{g_0}\Big(u_\nu-\bar{u}_{z_\nu}-\sum_{k=1}^m\alpha_{k,\nu}u_{(x_{k,\nu},\varepsilon_{k,\nu})}\Big)d\mu_{g_0}\\
&\leq \int_M\Big(u_\nu-\bar{u}_z-\sum_{k=1}^m\alpha_{k}u_{(x_{k},\varepsilon_{k})}\Big)P_\gamma^{g_0}\Big(u_\nu-\bar{u}_z-\sum_{k=1}^m\alpha_{k}u_{(x_{k},\varepsilon_{k})}\Big)d\mu_{g_0}
\end{split}
\end{equation}
for all $\big(z,(x_k,\varepsilon_k,\alpha_k)_{1\leq k\leq m}\big)\in\mathcal{A}_\nu$.

\begin{prop}\label{prop6.6}
(i) For all $i\neq j$, we have
$$\frac{\varepsilon_{i,\nu}}{\varepsilon_{j,\nu}}+\frac{\varepsilon_{j,\nu}}{\varepsilon_{i,\nu}}
+\frac{d(x_{i,\nu},x_{j,\nu})}{\varepsilon_{i,\nu}\varepsilon_{j,\nu}}\to\infty
$$
as $\nu\to\infty$.\\
(ii) We have
$$u_\nu-\bar{u}_{z_\nu}-\sum_{k=1}^m\alpha_{k,\nu}u_{(x_{k,\nu},\varepsilon_{k,\nu})}\to 0~~\mbox{ in }H^\gamma(M)$$
as $\nu\to\infty$.
\end{prop}
\begin{proof}
Part (i) is identical to the proof of \cite[Proposition 6.6(i)]{Brendle}.
For Part (ii), it follows from definition of $\big(z_\nu,(x_{k,\nu},\varepsilon_{k,\nu},\alpha_{k,\nu})_{1\leq k\leq m}\big)$
that
\begin{equation*}
\begin{split}
&\int_M\Big(u_\nu-\bar{u}_{z_\nu}-\sum_{k=1}^m\alpha_{k,\nu}u_{(x_{k,\nu},\varepsilon_{k,\nu})}\Big)P_\gamma^{g_0}\Big(u_\nu-\bar{u}_{z_\nu}-\sum_{k=1}^m\alpha_{k,\nu}u_{(x_{k,\nu},\varepsilon_{k,\nu})}\Big)d\mu_{g_0}\\
&\leq \int_M\Big(u_\nu-u_\infty-\sum_{k=1}^m u_{(x_{k,\nu}^*,\varepsilon_{k,\nu}^*)}\Big)P_\gamma^{g_0}\Big(u_\nu-u_\infty-\sum_{k=1}^m u_{(x_{k,\nu}^*,\varepsilon_{k,\nu}^*)}\Big)d\mu_{g_0}.
\end{split}
\end{equation*}
By Lemma \ref{lem2.1}, the expression on the right hand side tends to $0$
as $\nu\to\infty$. This proves the assertion.
\end{proof}

\begin{prop}\label{prop6.7}
We have
$$|z_\nu|=o(1)$$
and
$$d(x_k,x_{k,\nu}^*)=o(1)\varepsilon_{k,\nu}^*,~~\frac{\varepsilon_k}{\varepsilon_{k,\nu}^*}=1+o(1),~~
\alpha_k=1+o(1)$$
for all $1\leq k\leq m$. In particular,
$\big(z_\nu,(x_{k,\nu},\varepsilon_{k,\nu},\alpha_{k,\nu})_{1\leq k\leq m}\big)$
is an interior point of $\mathcal{A}_\nu$ if $\nu$ is sufficiently large. As a consequence, we have
\begin{equation}\label{eq: Linfty of u_z_nu}
    \|\bar{u}_{z_\nu}\|_{L^\infty(M)}\leq C,
\end{equation}
with a constant $C>0$ independent to $\nu$.
\end{prop}
\begin{proof}
We find
\begin{equation*}
\begin{split}
&\Bigg\|\bar{u}_{z_\nu}+\sum_{k=1}^m\alpha_{k,\nu}u_{(x_{k,\nu},\varepsilon_{k,\nu})}-u_\infty-\sum_{k=1}^m u_{(x_{k,\nu}^*,\varepsilon_{k,\nu}^*)}\Bigg\|_{H^\gamma(M)}\\
&\leq \Bigg\|u_\nu-\bar{u}_{z_\nu}-\sum_{k=1}^m\alpha_{k,\nu}u_{(x_{k,\nu},\varepsilon_{k,\nu})}\Bigg\|_{H^\gamma(M)}
+\Bigg\|u_\nu-u_\infty-\sum_{k=1}^m u_{(x_{k,\nu}^*,\varepsilon_{k,\nu}^*)}\Bigg\|_{H^\gamma(M)}\\
&=o(1)
\end{split}
\end{equation*}
by Lemma \ref{lem2.1} and Proposition \ref{prop6.6}.
From this, 
\begin{align*}
    |z_\nu|&\leq C\sup_{a\in A}\left|\int_M(\bar{u}_{z_\nu}-u_\infty)P_\gamma^{g_0}\psi_a\,d\mu_{g_0}\right|\\
    &\leq C\|\bar{u}_{z_\nu}-u_\infty\|_{H^\gamma(M)}\\
    &\leq C\Bigg\|\bar{u}_{z_\nu}+\sum_{k=1}^m\alpha_{k,\nu}u_{(x_{k,\nu},\varepsilon_{k,\nu})}-u_\infty-\sum_{k=1}^m u_{(x_{k,\nu}^*,\varepsilon_{k,\nu}^*)}\Bigg\|_{H^\gamma(M)}+o(1),
\end{align*}
and the assertion follows.
\end{proof}

As in the previous section, we assume that
$$\varepsilon_{i,\nu}\leq \varepsilon_{j,\nu}~~\mbox{ for }i\leq j.$$
We now decompose the function $u_\nu$ as
$$u_\nu=v_\nu+w_\nu,$$
where
$$v_\nu=\bar{u}_{z_\nu}+\sum_{k=1}^m\alpha_{k,\nu}u_{(x_{k,\nu},\varepsilon_{k,\nu})},$$
and
$$w_\nu=u_\nu-\bar{u}_{z_\nu}-\sum_{k=1}^m\alpha_{k,\nu}u_{(x_{k,\nu},\varepsilon_{k,\nu})}.$$
It follows from Proposition \ref{prop6.6} that
$$\int_Mw_\nu P_\gamma^{g_0} w_\nu d\mu_{g_0}=o(1).$$
In the following, we repeat similar statements as Proposition \ref{prop: some estimates on u^{n+2gamma/n-2gamma}w_nu} and Proposition \ref{prop: estimate on u^{4gamma/n-2gamma}w_nu^2}.

\begin{prop}\label{prop: estimates on w_nu}
(i) For every $a\in A$, we have 
$$\left|\int_Mu_\infty^{\frac{4\gamma}{n-2\gamma}}\psi_a w_\nu d\mu_{g_0}\right|\leq o(1)\|w_\nu\|_{L^1(M)}.$$
(ii) For every $1\leq k\leq m$, we have
$$\left|\int_Mu_{(x_{k,\nu},\varepsilon_{k,\nu})}^{\frac{n+2\gamma}{n-2\gamma}}w_\nu d\mu_{g_0}\right|\leq o(1)\|w_\nu\|_{H^\gamma(M)}.$$
(iii) For every $1\leq k\leq m$, we have
$$\left|\int_Mu_{(x_{k,\nu},\varepsilon_{k,\nu})}^{\frac{n+2\gamma}{n-2\gamma}}\frac{\varepsilon_{k,\nu}^2-d(x_{k,\nu},x)^2}{\varepsilon_{k,\nu}^2+d(x_{k,\nu},x)^2}w_\nu d\mu_{g_0}\right|\leq o(1)\|w_\nu\|_{H^\gamma(M)}.$$
(iv) For all $1\leq k\leq m$,
$$\left|\int_Mu_{(x_{k,\nu},\varepsilon_{k,\nu})}^{\frac{n+2\gamma}{n-2\gamma}}\frac{\varepsilon_{k,\nu}\exp^{-1}_{x_{k,\nu}}(x)}{\varepsilon_{k,\nu}^2+d(x_{k,\nu},x)^2}w_\nu d\mu_{g_0}\right|\leq o(1)\|w_\nu\|_{H^\gamma(M)}.$$
\end{prop}
\begin{proof}
(i) In the view of Lemma \ref{lem: def of perturbed uinfty}, define
\[
    \tilde{\psi}_{a,z}:=\frac{\partial}{\partial z_a}\bar{u}_z\bigg.
\]
Then from the definition of $z_\nu$, we have
\begin{equation}\label{eq: L1 estimate proof 1}
    \int_M w_\nu P_\gamma^{g_0}\tilde{\psi}_{a,z_\nu}\,d\mu_{g_0}=0.
\end{equation}
Also, from the definition of $\bar{u}_z$,
\begin{equation}\label{eq: L1 estimate proof 2}
    \int_M u_\infty^{\frac{4\gamma}{n-2\gamma}}\tilde{\psi}_{a,z}\psi_a\,d\mu_{g_0}=1,
\end{equation}
and for all $a'\in A$ such that $a'\neq a$,
\begin{equation}\label{eq: L1 estimate proof 3}
    \int_M u_\infty^{\frac{4\gamma}{n-2\gamma}}\tilde{\psi}_{a,z}\psi_{a'}\,d\mu_{g_0}=0.
\end{equation}
On the other hand, for $a'\notin A$ from Lemma \ref{lem6.3}, 
\[
    \begin{split}
        0&=\frac{\partial}{\partial z_a}\Pi\left(P_\gamma^{g_0}\bar{u}_{z_\nu}-s_\infty \bar{u}_{z_\nu}^{\frac{n+2\gamma}{n-2\gamma}}\right)\\
        &=\sum_{a'\not\in A}\left(\int_M \psi_{a'}\left(P_\gamma^{g_0}\tilde{\psi}_{a,z_\nu}-s_\infty\frac{n+2\gamma}{n-2\gamma} \bar{u}_{z_\nu}^{\frac{4\gamma}{n-2\gamma}}\tilde{\psi}_{a,z_\nu}\right)\,d\mu_{g_0}\right)\psi_{a'}u_\infty^{\frac{4\gamma}{n-2\gamma}},
    \end{split}
\]
thus, it follows that for any $a'\notin A$,
\[
    s_\infty\frac{n+2\gamma}{n-2\gamma}\int_M\bar{u}_{z_\nu}^{\frac{4\gamma}{n-2\gamma}}\psi_{a'}\tilde{\psi}_{a,z_\nu}\,d\mu_{g_0}=\lambda_{a'}\int_Mu_\infty^{\frac{4\gamma}{n-2\gamma}}\psi_{a'}\tilde{\psi}_{a,z_\nu} \, d\mu_{g_0}.
\]
Note that $\lambda_{a'}>s_\infty\frac{n+2\gamma}{n-2\gamma}$. Also, from Proposition \ref{prop6.7}, $\bar{u}_{z_\nu}=u_\infty+o(1)$. Combine these facts, we obtain
\begin{equation}\label{eq: L1 estimate proof 4}
    \int_M \bar{u}_\infty^{\frac{4\gamma}{n-2\gamma}}\tilde{\psi}_{a,z_\nu}\psi_{a'}\,d\mu_{g_0}=o(1),\quad\forall a'\notin A.
\end{equation}
From \eqref{eq: L1 estimate proof 2}, \eqref{eq: L1 estimate proof 3}, and \eqref{eq: L1 estimate proof 4}, it follows that
\[
    \|P_\gamma^{g_0}(\psi_a-\tilde{\psi}_{a,z_\nu})\|_{L^\infty(M)}=o(1).
\]
Finally, with \eqref{eq: L1 estimate proof 1}, we have
\[
    \begin{split}
        \lambda_a\int_Mu_\infty^{\frac{4\gamma}{n-2\gamma}}\psi_aw_\nu\,d\mu_{g_0}&=\int_M w_\nu P_\gamma^{g_0}\psi_a\\
        &=\int_M w_\nu P_\gamma^{g_0}(\psi_a-\tilde{\psi}_{a,z_\nu})\,d\mu_{g_0}\\
        &\leq \|P_\gamma^{g_0}(\psi_a-\tilde{\psi}_{a,z_\nu})\|_{L^\infty(M)}\|w_\nu\|_{L^1(M)}\\
        &=o(1)\|w_\nu\|_{L^1(M)}.
    \end{split}
\]
\\
(ii) By definition of  $\big(z_\nu,(x_{k,\nu},\varepsilon_{k,\nu},\alpha_{k,\nu})_{1\leq k\leq m}\big)$, we have
$$\int_Mw_\nu P_\gamma^{g_0}u_{(x_{k,\nu},\varepsilon_{k,\nu})} d\mu_{g_0}=0,$$
for all $1\leq k\leq m$. Using the estimate \eqref{eq: L 2n / n+2gamma estimate on bubble}
$$\Big\|P_\gamma^{g_0}u_{(x_{k,\nu},\varepsilon_{k,\nu})}-s_\infty u_{(x_{k,\nu},\varepsilon_{k,\nu})}^{\frac{n+2\gamma}{n-2\gamma}}\Big\|_{L^{\frac{2n}{n+2\gamma}}(M)}=o(1),$$
we conclude that
\begin{equation*}
\begin{split}
&\left|s_\infty \int_Mu_{(x_{k,\nu},\varepsilon_{k,\nu})}^{\frac{n+2\gamma}{n-2\gamma}}w_\nu d\mu_{g_0}\right|\\
&=\left|\int_M\Big(P_\gamma^{g_0}u_{(x_{k,\nu},\varepsilon_{k,\nu})}-s_\infty u_{(x_{k,\nu},\varepsilon_{k,\nu})}^{\frac{n+2\gamma}{n-2\gamma}}\Big)w_\nu d\mu_{g_0}\right|\\
&\leq \Big\|P_\gamma^{g_0}u_{(x_{k,\nu},\varepsilon_{k,\nu})}-s_\infty u_{(x_{k,\nu},\varepsilon_{k,\nu})}^{\frac{n+2\gamma}{n-2\gamma}}\Big\|_{L^{\frac{2n}{n+2\gamma}}(M)}
\|w_\nu\|_{L^{\frac{2n}{n-2\gamma}}(M)}\\
&\leq o(1)\|w_\nu\|_{H^\gamma(M)}
\end{split}
\end{equation*}
for all $1\leq k\leq m$.

The remaining statements follow similarly.
\end{proof}

\begin{prop}\label{prop6.9}
If $\nu$ is sufficiently large, then we have
\begin{equation*}
\frac{n+2\gamma}{n-2\gamma}s_\infty\int_M\Big(u_\infty^{\frac{4\gamma}{n-2\gamma}}+\sum_{j=1}^m  u_{(x_{k,\nu},\varepsilon_{k,\nu})}^{\frac{4\gamma}{n-2\gamma}}\Big)w_\nu^2 d\mu_{g_0}
\leq (1-c)\int_M w_\nu P_\gamma^{g_0} w_\nu d\mu_{g_0}
\end{equation*}
for some positive constant $c$ independent of $\nu$.
\end{prop}
\begin{proof}
Suppose it were not true. Upon rescaling, we obtain a subsequence of functions $\{\tilde{w}_\nu:\nu\in\mathbb{N}\}$ such that
\begin{equation}\label{123}
\int_M\widetilde{w}_\nu P_\gamma^{g_0} \widetilde{w}_\nu d\mu_{g_0}=1
\end{equation}
and
\begin{equation}\label{124}
\lim_{\nu\to\infty}\frac{n+2\gamma}{n-2\gamma}s_\infty\int_M\Big(u_\infty^{\frac{4\gamma}{n-2\gamma}}+\sum_{j=1}^m  u_{(x_{k,\nu},\varepsilon_{k,\nu})}^{\frac{4\gamma}{n-2\gamma}}\Big)\widetilde{w}_\nu^2 d\mu_{g_0}
\geq 1.
\end{equation}
From (\ref{123}), we have
\begin{equation}\label{125}
\int_M|\widetilde{w}_\nu|^{\frac{2n}{n-2\gamma}} d\mu_{g_0}
\leq Y_\gamma(M,[g_0])^{-\frac{n}{n-2\gamma}}.
\end{equation}
In view of Lemma \ref{lem2.1}, there exists a sequence $\{N_\nu:\nu\in\mathbb{N}\}$
such that $N_\nu\to\infty$, $N_\nu\varepsilon_{j,\nu}\to 0$ for all $1\leq j\leq m$, and
\begin{equation}\label{126}
\frac{1}{N_\nu}\frac{\varepsilon_{j,\nu}+d(x_{i,\nu},x_{j,\nu})}{\varepsilon_{i,\nu}}\to\infty
\end{equation}
for all $i<j$. Let
$$\Omega_{j,\nu}=B_{N_\nu\varepsilon_{j,\nu}}(x_{j,\nu})\setminus\bigcup_{i=1}^{j-1} B_{N_\nu\varepsilon_{i,\nu}}(x_{i,\nu})$$
for every $1\leq j\leq m$. In view of (\ref{123}) and (\ref{124}),
there are only two possibilities:\\
Case 1. Suppose that
\begin{equation}\label{127}
\lim_{\nu\to\infty}\int_M u_\infty^{\frac{4\gamma}{n-2\gamma}}\widetilde{w}_\nu^2 d\mu_{g_0}>0
\end{equation}
and
\begin{equation}\label{128}
\lim_{\nu\to\infty}\int_{M\setminus\cup_{j=1}^m\Omega_{j,\nu}}\widetilde{w}_\nu P_\gamma^{g_0} \widetilde{w}_\nu d\mu_{g_0}
\leq \lim_{\nu\to\infty}\frac{n+2\gamma}{n-2\gamma}s_\infty\int_M u_\infty^{\frac{4\gamma}{n-2\gamma}}\widetilde{w}_\nu^2 d\mu_{g_0}.
\end{equation}
Let $\widetilde{w}$ be the weak limit of the sequence $\{\widetilde{w}_\nu:\nu\in\mathbb{N}\}$.
Then the function $\widetilde{w}$ satisfies
$$\int_M u_\infty^{\frac{4\gamma}{n-2\gamma}}\widetilde{w}^2 d\mu_{g_0}>0$$
and
$$\int_M\widetilde{w} P_\gamma^{g_0} \widetilde{w} d\mu_{g_0}
\leq \frac{n+2\gamma}{n-2\gamma}s_\infty\int_M u_\infty^{\frac{4\gamma}{n-2\gamma}}\widetilde{w}^2 d\mu_{g_0}.$$
Expanding $\tilde{w}$ in orthogonal basis $(\psi_a)$, this implies 
$$\sum_{a\in\mathbb{N}}\lambda_a \left(\int_Mu_\infty^{\frac{4\gamma}{n-2\gamma}}\psi_a\widetilde{w}d\mu_{g_0}\right)^2
\leq\sum_{a\in\mathbb{N}}\frac{n+2\gamma}{n-2\gamma}s_\infty\left(\int_Mu_\infty^{\frac{4\gamma}{n-2\gamma}}\psi_a\widetilde{w}d\mu_{g_0}\right).$$
Using Proposition \ref{prop: estimates on w_nu}, we obtain
$$\int_Mu_\infty^{\frac{4\gamma}{n-2\gamma}}\psi_a\widetilde{w}d\mu_{g_0}=0$$
for all $a\in A$. Thus, we conclude that $\widetilde{w}\equiv 0$, which is a contradiction.\\
Case 2. Suppose there exists an integer $1\leq j\leq m$ such that
\begin{equation}\label{129}
\lim_{\nu\to\infty}\int_M u_{(x_{j,\nu},\varepsilon_{j,\nu})}^{\frac{4\gamma}{n-2\gamma}}\widetilde{w}_\nu^2 d\mu_{g_0}>0
\end{equation}
and
\begin{equation}\label{130}
\lim_{\nu\to\infty}\int_{\Omega_{j,\nu}}\widetilde{w}_\nu P_\gamma^{g_0} \widetilde{w}_\nu d\mu_{g_0}
\leq \lim_{\nu\to\infty}\frac{n+2\gamma}{n-2\gamma}s_\infty\int_M u_{(x_{j,\nu},\varepsilon_{j,\nu})}^{\frac{4\gamma}{n-2\gamma}}\widetilde{w}_\nu^2 d\mu_{g_0}.
\end{equation}
We now define a sequence of functions $\hat{w}_\nu: T_{x_{j,\nu}}M\to\mathbb{R}$ by
$$\hat{w}_\nu(\xi)=\varepsilon_{j,\nu}^{\frac{n-2\gamma}{2}}\widetilde{w}_\nu(\exp_{x_{j,\nu}}(\varepsilon_{j,\nu}\xi))$$
for $\xi\in T_{x_{j,\nu}}M$.
The sequence $\{\hat{w}_\nu:\nu\in\mathbb{N}\}$ satisfies
$$\lim_{\nu\to\infty}\int_{\{\xi\in T_{x_{j,\nu}}M: |\xi|\leq N_\nu\}}\hat{w}_\nu(\xi)(-\Delta_{\mathbb{R}^n})^\gamma \hat{w}_\nu(\xi)d\xi
\leq 1$$
and
$$\lim_{\nu\to\infty}\int_{\{\xi\in T_{x_{j,\nu}}M: |\xi|\leq N_\nu\}}|\hat{w}_\nu(\xi)|^{\frac{2n}{n-2\gamma}}d\xi\leq Y_\gamma(M,[g_0])^{-\frac{n}{n-2\gamma}}.$$
Hence, if we take the weak limit as $\nu\to\infty$
such that $\hat{w}_\nu \rightharpoonup \hat{w}$, we then obtain a function $\hat{w}:\mathbb{R}^n\to\mathbb{R}$
satisfying
\begin{equation}\label{5.10}
\int_{\mathbb{R}^n}\frac{\hat{w}_\nu(\xi)^2}{(1+|\xi|^2)^{2\gamma}}d\xi>0
\end{equation}
and
\begin{equation}\label{5.11}
\int_{\mathbb{R}^n}\hat{w}(\xi) (-\Delta_{\mathbb{R}^n})^\gamma\hat{w}(\xi)d\xi\leq\alpha_{n,\gamma}^{\frac{4\gamma}{n-2\gamma}}\frac{n+2\gamma}{n-2\gamma}
\int_{\mathbb{R}^n}\frac{\hat{w}(\xi)^2}{(1+|\xi|^2)^{2\gamma}}d\xi
\end{equation}
where $\alpha_{n\gamma}$ is a constant in (\ref{3.11}).
Moreover, it follows from Proposition \ref{prop: estimates on w_nu} that
\begin{equation}\label{5.12}
\begin{split}
\int_{\mathbb{R}^n}\left(\frac{1}{1+|\xi|^2}\right)^{\frac{n+2\gamma}{2}}\hat{w}(\xi)d\xi&=0,\\
\int_{\mathbb{R}^n}\left(\frac{1}{1+|\xi|^2}\right)^{\frac{n+2\gamma}{2}}\frac{1-|\xi|^2}{1+|\xi|^2}\hat{w}(\xi)d\xi&=0,\\
\int_{\mathbb{R}^n}\left(\frac{1}{1+|\xi|^2}\right)^{\frac{n+2\gamma}{2}}\frac{\xi}{1+|\xi|^2}\hat{w}(\xi)d\xi&=0.
\end{split}
\end{equation}
Since $\hat{w}$ satisfies (\ref{5.11}) and (\ref{5.12}),
we can follow the argument in the proof of \cite[Lemma 5.1]{ChanSireSun}
to conclude that $\hat{w}\equiv 0$, which contradicts to (\ref{5.10}).
This proves the proposition.
\end{proof}

\begin{cor}\label{cor6.10}
 If $\nu$ is sufficiently large, then we have
\begin{equation*}
\frac{n+2\gamma}{n-2\gamma}s_\infty\int_Mv_\nu^{\frac{4\gamma}{n-2\gamma}}w_\nu^2 d\mu_{g_0}
\leq (1-c)\int_M w_\nu P_\gamma^{g_0} w_\nu d\mu_{g_0}
\end{equation*}
for some positive constant $c$ independent of $\nu$.
\end{cor}
\begin{proof}
By definition of $v_\nu$, we have
$$\int_M\Big|v_\nu^{\frac{4\gamma}{n-2\gamma}}-u_\infty^{\frac{4\gamma}{n-2\gamma}}-\sum_{j=1}^m  u_{(x_{k,\nu},\varepsilon_{k,\nu})}^{\frac{4\gamma}{n-2\gamma}}\Big|^{\frac{n}{2\gamma}}d\mu_{g_0}=o(1).$$
Hence, the assertion follows from Proposition \ref{prop6.9}.
\end{proof}

\begin{lem}\label{lem6.11}
 The difference $u_\nu-\bar{u}_{z_\nu}$ satisfies the estimate
$$\|u_\nu-\bar{u}_{z_\nu}\|_{L^{\frac{n+2\gamma}{n-2\gamma}}(M)}^{\frac{n+2\gamma}{n-2\gamma}}
\leq C\Big\|u_\nu^{\frac{n+2\gamma}{n-2\gamma}}(R_\gamma^{g_\nu}-s_\infty)\Big\|^{\frac{n+2\gamma}{n-2\gamma}}_{L^{\frac{2n}{n+2\gamma}}(M)}+ C\sum_{k=1}^m\varepsilon_{k,\nu}^{\frac{n-2\gamma}{2}}$$
if $\nu$ is sufficiently large.
\end{lem}
\begin{proof}
Using the identities
$$P_\gamma^{g_0}u_\nu-s_\infty u_\nu^{\frac{n+2\gamma}{n-2\gamma}}
=(R_\gamma^{g_\nu}-s_\infty)u_\nu^{\frac{n+2\gamma}{n-2\gamma}}$$
and
$$\Pi\Big(P_\gamma^{g_0}\bar{u}_{z_\nu}-s_\infty \bar{u}_{z_\nu}^{\frac{n+2\gamma}{n-2\gamma}}\Big)=0,$$
we obtain
\begin{equation}\label{134}
\begin{split}
&\Pi\left(P_\gamma^{g_0}(u_\nu-\bar{u}_{z_\nu})-\frac{n+2\gamma}{n-2\gamma}s_\infty u_\infty^{\frac{4\gamma}{n-2\gamma}}(u_\nu-\bar{u}_{z_\nu})\right)\\
&=\Pi\Bigg((R_\gamma^{g_\nu}-s_\infty)u_\nu^{\frac{n+2\gamma}{n-2\gamma}}+\frac{n+2\gamma}{n-2\gamma}s_\infty (\bar{u}_\nu^{\frac{4\gamma}{n-2\gamma}}- u_\infty^{\frac{4\gamma}{n-2\gamma}})(u_\nu-\bar{u}_{z_\nu})\\
&\hspace{8mm}
-s_\infty\Big(\bar{u}_{z_\nu}^{\frac{n+2\gamma}{n-2\gamma}}+\frac{n+2\gamma}{n-2\gamma} \bar{u}_{z_\nu}^{\frac{4\gamma}{n-2\gamma}}(u_\nu-\bar{u}_{z_\nu})-u_\nu^{\frac{n+2\gamma}{n-2\gamma}}\Big)\Bigg).
\end{split}
\end{equation}
Applying Lemma \ref{lem6.3}, (i) with $f=u_\nu-\bar{u}_{z_\nu}$ gives
\begin{equation*}
\begin{split}
&\|u_\nu-\bar{u}_{z_\nu}\|_{L^{\frac{n+2\gamma}{n-2\gamma}}(M)}\\
&\hspace{4mm}\leq C\left\|\Pi\left(P_\gamma^{g_0}(u_\nu-\bar{u}_{z_\nu})-\frac{n+2\gamma}{n-2\gamma}s_\infty u_\infty^{\frac{4\gamma}{n-2\gamma}}(u_\nu-\bar{u}_{z_\nu})\right)\right\|_{L^{\frac{n(n+2\gamma)}{n^2+4\gamma^2}}(M)}\\
&\hspace{8mm}+C\sup_{a\in A}\left|\int_M u_\infty^{\frac{4\gamma}{n-2\gamma}}\psi_a (u_\nu-\bar{u}_{z_\nu})d\mu_{g_0}\right|.
\end{split}
\end{equation*}
Combining this with (\ref{134}), we obtain
\begin{equation}\label{136}
\begin{split}
&\|u_\nu-\bar{u}_{z_\nu}\|_{L^{\frac{n+2\gamma}{n-2\gamma}}(M)}\\
&\hspace{4mm}\leq C\Big\|(R_\gamma^{g_\nu}-s_\infty)u_\nu^{\frac{n+2\gamma}{n-2\gamma}}\Big\|_{L^{\frac{n(n+2\gamma)}{n^2+4\gamma^2}}(M)}\\
&\hspace{8mm}+C\Big\|(\bar{u}_\nu^{\frac{4\gamma}{n-2\gamma}}- u_\infty^{\frac{4\gamma}{n-2\gamma}})(u_\nu-\bar{u}_{z_\nu})\Big\|_{L^{\frac{n(n+2\gamma)}{n^2+4\gamma^2}}(M)}\\
&\hspace{8mm}+C\left\|\bar{u}_{z_\nu}^{\frac{n+2\gamma}{n-2\gamma}}+\frac{n+2\gamma}{n-2\gamma} \bar{u}_{z_\nu}^{\frac{4\gamma}{n-2\gamma}}(u_\nu-\bar{u}_{z_\nu})-u_\nu^{\frac{n+2\gamma}{n-2\gamma}}\right\|_{L^{\frac{n(n+2\gamma)}{n^2+4\gamma^2}}(M)}\\
&\hspace{8mm}+C\sup_{a\in A}\left|\int_M u_\infty^{\frac{4\gamma}{n-2\gamma}}\psi_a (u_\nu-\bar{u}_{z_\nu})d\mu_{g_0}\right|.
\end{split}
\end{equation}
Using the pointwise estimate \eqref{eq: pointwise estimate 2}
\begin{equation}\label{137}
\begin{split}
&\left|\bar{u}_{z_\nu}^{\frac{n+2\gamma}{n-2\gamma}}+\frac{n+2\gamma}{n-2\gamma} \bar{u}_{z_\nu}^{\frac{4\gamma}{n-2\gamma}}(u_\nu-\bar{u}_{z_\nu})-u_\nu^{\frac{n+2\gamma}{n-2\gamma}}\right|\\
&\leq C\bar{u}_{z_\nu}^{\max\{0,\frac{4\gamma}{n-2\gamma}-1\}}|u_\nu-\bar{u}_{z_\nu}|^{\min\{\frac{n+2\gamma}{n-2\gamma},2\}}
+|u_\nu-\bar{u}_{z_\nu}|^{\frac{n+2\gamma}{n-2\gamma}}
\end{split}
\end{equation}
and \eqref{eq: Linfty of u_z_nu}, we obtain
\begin{equation*}
\begin{split}
&\left\|\bar{u}_{z_\nu}^{\frac{n+2\gamma}{n-2\gamma}}+\frac{n+2\gamma}{n-2\gamma} \bar{u}_{z_\nu}^{\frac{4\gamma}{n-2\gamma}}(u_\nu-\bar{u}_{z_\nu})-u_\nu^{\frac{n+2\gamma}{n-2\gamma}}\right\|_{L^{\frac{n(n+2\gamma)}{n^2+4\gamma^2}}(M)}\\
&\leq C\Big\||u_\nu-\bar{u}_{z_\nu}|^{\min\{\frac{n+2\gamma}{n-2\gamma},2\}}
+|u_\nu-\bar{u}_{z_\nu}|^{\frac{n+2\gamma}{n-2\gamma}}\Big\|_{L^{\frac{n(n+2\gamma)}{n^2+4\gamma^2}}(M)}.
\end{split}
\end{equation*}
Note that
\begin{equation*}
\begin{split}
&\Big\||u_\nu-\bar{u}_{z_\nu}|^{\min\{\frac{n+2\gamma}{n-2\gamma},2\}}
+|u_\nu-\bar{u}_{z_\nu}|^{\frac{n+2\gamma}{n-2\gamma}}\Big\|_{L^{\frac{n(n+2\gamma)}{n^2+4\gamma^2}}(M)}\\
&\leq \Big\||u_\nu-\bar{u}_{z_\nu}|^{\min\{\frac{n+2\gamma}{n-2\gamma},2\}}
+|u_\nu-\bar{u}_{z_\nu}|^{\frac{n+2\gamma}{n-2\gamma}}\Big\|_{L^{\frac{n(n+2\gamma)}{n^2+4\gamma^2}}(\bigcup_{k=1}^mB_{N\varepsilon_{k,\nu}}(x_{k,\nu}))}\\
&\hspace{4mm}+\Big\||u_\nu-\bar{u}_{z_\nu}|^{\min\{\frac{n+2\gamma}{n-2\gamma},2\}}
+|u_\nu-\bar{u}_{z_\nu}|^{\frac{n+2\gamma}{n-2\gamma}}\Big\|_{L^{\frac{n(n+2\gamma)}{n^2+4\gamma^2}}(M\setminus\bigcup_{k=1}^mB_{N\varepsilon_{k,\nu}}(x_{k,\nu}))}\\
&\leq C\sum_{k=1}^m(N\varepsilon_{k,\nu})^{\frac{(n-2\gamma)^2}{2(n+2\gamma)}}\Big\||u_\nu-\bar{u}_{z_\nu}|^{\min\{\frac{n+2\gamma}{n-2\gamma},2\}}
+|u_\nu-\bar{u}_{z_\nu}|^{\frac{n+2\gamma}{n-2\gamma}}\Big\|_{L^{\frac{2n}{n+2\gamma}}(M)}\\
&\hspace{4mm}+C\Big\||u_\nu-\bar{u}_{z_\nu}|^{\min\{\frac{4\gamma}{n-2\gamma},1\}}
+|u_\nu-\bar{u}_{z_\nu}|^{\frac{4\gamma}{n-2\gamma}}\Big\|_{L^{\frac{n}{2\gamma}}(M\setminus\bigcup_{k=1}^mB_{N\varepsilon_{k,\nu}}(x_{k,\nu}))}\|u_\nu-\bar{u}_{z_\nu}\|_{L^{\frac{n+2\gamma}{n-2\gamma}}(M)}
\end{split}
\end{equation*}
by H\"{o}lder's inequality.
Since 
\begin{equation*}
\begin{split}
&\|u_\nu-\bar{u}_{z_\nu}\|_{L^{\frac{2n}{n-2\gamma}}(M\setminus\bigcup_{k=1}^mB_{N\varepsilon_{k,\nu}}(x_{k,\nu}))}\\
&=\Big\|\sum_{k=1}^m\alpha_{k,\nu}u_{(x_{k,\nu},\varepsilon_{k,\nu})}+w_\nu\Big\|_{L^{\frac{2n}{n-2\gamma}}(M\setminus\bigcup_{k=1}^mB_{N\varepsilon_{k,\nu}}(x_{k,\nu}))}\\
&\leq \sum_{k=1}^m\alpha_{k,\nu}\|u_{(x_{k,\nu},\varepsilon_{k,\nu})}\|_{L^{\frac{2n}{n-2\gamma}}(M\setminus\bigcup_{k=1}^mB_{N\varepsilon_{k,\nu}}(x_{k,\nu}))}+
\|w_\nu\|_{L^{\frac{2n}{n-2\gamma}}(M)}\\
&\leq CN^{-\frac{n-2\gamma}{2}}+o(1),
\end{split}
\end{equation*}
it follows that
\begin{equation*}
\begin{split}
&\left\|\bar{u}_{z_\nu}^{\frac{n+2\gamma}{n-2\gamma}}+\frac{n+2\gamma}{n-2\gamma} \bar{u}_{z_\nu}^{\frac{4\gamma}{n-2\gamma}}(u_\nu-\bar{u}_{z_\nu})-u_\nu^{\frac{n+2\gamma}{n-2\gamma}}\right\|_{L^{\frac{n(n+2\gamma)}{n^2+4\gamma^2}}(M)}\\
&\leq C\sum_{k=1}^m(N\varepsilon_{k,\nu})^{\frac{(n-2\gamma)^2}{2(n+2\gamma)}}
+\big( CN^{-\frac{n-2\gamma}{2}}+o(1)\big)\|u_\nu-\bar{u}_{z_\nu}\|_{L^{\frac{n+2\gamma}{n-2\gamma}}(M)}.
\end{split}
\end{equation*}
Moreover, we have 
\[
    \int_M \psi_a u_\infty^{\frac{4\gamma}{n-2\gamma}}u_{(x_{k,\nu},\varepsilon_{k,\nu})}\,d\mu_{g_0}\leq \|\psi_au_\infty^{\frac{4\gamma}{n-2\gamma}}\|_{L^\infty}\int_M u_{(x_{k,\nu},\varepsilon_{k,\nu})}\,d\mu_{g_0}\leq C\varepsilon_{k,\nu}^{\frac{n-2\gamma}{2}}.
\]
Thus, using \eqref{prop: estimates on w_nu}, we obtain
\begin{equation}\label{142}
\begin{split}
&\sup_{a\in A}\left|\int_M u_\infty^{\frac{4\gamma}{n-2\gamma}}\psi_a (u_\nu-\bar{u}_{z_\nu})d\mu_{g_0}\right|\\
&=\sup_{a\in A}\left|\int_M u_\infty^{\frac{4\gamma}{n-2\gamma}}\psi_a \Big(\sum_{k=1}^m\alpha_{k,\nu}u_{(x_{k,\nu},\varepsilon_{k,\nu})}+w_\nu\Big)d\mu_{g_0}\right|\\
&\leq C\sum_{k=1}^m\varepsilon_{k,\nu}^{\frac{n-2\gamma}{2}}+o(1)\|w_\nu\|_{L^1(M)}\\
&\leq C\sum_{k=1}^m\varepsilon_{k,\nu}^{\frac{n-2\gamma}{2}}+o(1)\Big\|u_\nu-\bar{u}_{z_\nu}-\sum_{k=1}^m\alpha_{k,\nu}u_{(x_{k,\nu},\varepsilon_{k,\nu})}\Big\|_{L^1(M)}\\
&\leq C\sum_{k=1}^m\varepsilon_{k,\nu}^{\frac{n-2\gamma}{2}}+o(1) \|u_\nu-\bar{u}_{z_\nu}\|_{L^1(M)}.
\end{split}
\end{equation}
Putting these facts together, we conclude that
\begin{equation*}
\begin{split}
\|u_\nu-\bar{u}_{z_\nu}\|_{L^{\frac{n+2\gamma}{n-2\gamma}}(M)}
&\leq  C\Big\|(R_\gamma^{g_\nu}-s_\infty)u_\nu^{\frac{n+2\gamma}{n-2\gamma}}\Big\|_{L^{\frac{2n}{n+2\gamma}}(M)}\\
&\hspace{4mm}+C\sum_{k=1}^m(N\varepsilon_{k,\nu})^{\frac{(n-2\gamma)^2}{2(n+2\gamma)}}
+C\sum_{k=1}^m\varepsilon_{k,\nu}^{\frac{n-2\gamma}{2}}\\
&\hspace{4mm}+\big( CN^{-\frac{n-2\gamma}{2}}+o(1)\big)\|u_\nu-\bar{u}_{z_\nu}\|_{L^{\frac{n+2\gamma}{n-2\gamma}}(M)}.
\end{split}
\end{equation*}
Hence, if we choose $N$ sufficiently large, we then obtain
\begin{equation*}
\|u_\nu-\bar{u}_{z_\nu}\|_{L^{\frac{n+2\gamma}{n-2\gamma}}(M)}
\leq  C\Big\|(R_\gamma^{g_\nu}-s_\infty)u_\nu^{\frac{n+2\gamma}{n-2\gamma}}\Big\|_{L^{\frac{2n}{n+2\gamma}}(M)}+C\sum_{k=1}^m\varepsilon_{k,\nu}^{\frac{(n-2\gamma)^2}{2(n+2\gamma)}}.
\end{equation*}
From this, the assertion follows.
\end{proof}

\begin{lem}\label{lem6.12}
 The difference $u_\nu-\bar{u}_{z_\nu}$ satisfies the estimate
$$\|u_\nu-\bar{u}_{z_\nu}\|_{L^1(M)}
\leq C\Big\|(R_\gamma^{g_\nu}-s_\infty)u_\nu^{\frac{n+2\gamma}{n-2\gamma}}\Big\|_{L^{\frac{2n}{n+2\gamma}}(M)}+ C\sum_{k=1}^m\varepsilon_{k,\nu}^{\frac{n-2\gamma}{2}}$$
if $\nu$ is sufficiently large.
\end{lem}
\begin{proof}
Using Lemma \ref{lem6.3}, (ii) with $f=u_\nu-\bar{u}_{z_\nu}$, we have
\begin{equation*}
\begin{split}
\|u_\nu-\bar{u}_{z_\nu}\|_{L^{1}(M)}&\leq C\left\|\Pi\left(P_\gamma^{g_0}(u_\nu-\bar{u}_{z_\nu})-\frac{n+2\gamma}{n-2\gamma}s_\infty u_\infty^{\frac{4\gamma}{n-2\gamma}} (u_\nu-\bar{u}_{z_\nu})\right)\right\|_{L^1(M)}\\
&\hspace{4mm}+C\sup_{a\in A}\left|\int_M u_\infty^{\frac{4\gamma}{n-2\gamma}}\psi_a (u_\nu-\bar{u}_{z_\nu})d\mu_{g_0}\right|.
\end{split}
\end{equation*}
Combining this with (\ref{134}), we have
\begin{equation}\label{147}
\begin{split}
&\|u_\nu-\bar{u}_{z_\nu}\|_{L^1(M)}\\
&\hspace{4mm}\leq C\Big\|(R_\gamma^{g_\nu}-s_\infty)u_\nu^{\frac{n+2\gamma}{n-2\gamma}}\Big\|_{L^1(M)}\\
&\hspace{8mm}+C\Big\|(u_\nu^{\frac{4\gamma}{n-2\gamma}}- u_\infty^{\frac{4\gamma}{n-2\gamma}})(u_\nu-\bar{u}_{z_\nu})\Big\|_{L^1(M)}\\
&\hspace{8mm}+C\left\|\bar{u}_{z_\nu}^{\frac{n+2\gamma}{n-2\gamma}}+\frac{n+2\gamma}{n-2\gamma} \bar{u}_{z_\nu}^{\frac{4\gamma}{n-2\gamma}}(u_\nu-\bar{u}_{z_\nu})-u_\nu^{\frac{n+2\gamma}{n-2\gamma}}\right\|_{L^1(M)}\\
&\hspace{8mm}+C\sup_{a\in A}\left|\int_M u_\infty^{\frac{4\gamma}{n-2\gamma}}\psi_a (u_\nu-\bar{u}_{z_\nu})d\mu_{g_0}\right|.
\end{split}
\end{equation}
Using the pointwise estimate (\ref{137}) and H\"{o}lder's inequality, we get
\begin{equation*}
\begin{split}
&\left\|\bar{u}_{z_\nu}^{\frac{n+2\gamma}{n-2\gamma}}+\frac{n+2\gamma}{n-2\gamma} \bar{u}_{z_\nu}^{\frac{4\gamma}{n-2\gamma}}(u_\nu-\bar{u}_{z_\nu})-u_\nu^{\frac{n+2\gamma}{n-2\gamma}}\right\|_{L^1(M)}\\
&\leq C\Big\||u_\nu-\bar{u}_{z_\nu}|^{\min\{\frac{n+2\gamma}{n-2\gamma},2\}}
+|u_\nu-\bar{u}_{z_\nu}|^{\frac{n+2\gamma}{n-2\gamma}}\Big\|_{L^1(M)}\\
&\leq \|u_\nu-\bar{u}_{z_\nu}\|_{L^1(M)}^{\max\{0,1-\frac{n-2\gamma}{4}\}}
\big\||u_\nu-\bar{u}_{z_\nu}|^{\frac{n+2\gamma}{n-2\gamma}}\big\|_{L^1(M)}^{\min\{1,\frac{n-2\gamma}{4}\}}\\
&\hspace{4mm}
+C\big\||u_\nu-\bar{u}_{z_\nu}|^{\frac{n+2\gamma}{n-2\gamma}}\big\|_{L^1(M)}\\
&\leq \|u_\nu-\bar{u}_{z_\nu}\|_{L^1(M)}^{\max\{0,1-\frac{n-2\gamma}{4}\}}
\|u_\nu-\bar{u}_{z_\nu} \|_{L^{\frac{n+2\gamma}{n-2\gamma}}(M)}^{\frac{n+2\gamma}{n-2\gamma}\min\{1,\frac{n-2\gamma}{4}\}}\\
&\hspace{4mm}
+C\|u_\nu-\bar{u}_{z_\nu}\|_{L^{\frac{n+2\gamma}{n-2\gamma}}(M)}^{\frac{n+2\gamma}{n-2\gamma}}.
\end{split}
\end{equation*}
Combining this with (\ref{142}) and (\ref{147}), we conclude that
\begin{equation*}
\begin{split}
&\|u_\nu-\bar{u}_{z_\nu}\|_{L^1(M)}\\
&\hspace{4mm}\leq   C\Big\|(R_\gamma^{g_\nu}-s_\infty)u_\nu^{\frac{n+2\gamma}{n-2\gamma}}\Big\|_{L^{\frac{2n}{n+2\gamma}}(M)}\\
&\hspace{8mm}+
C\|u_\nu-\bar{u}_{z_\nu}\|_{L^1(M)}^{\max\{0,1-\frac{n-2\gamma}{4}\}}
\|u_\nu-\bar{u}_{z_\nu} \|_{L^{\frac{n+2\gamma}{n-2\gamma}}(M)}^{\frac{n+2\gamma}{n-2\gamma}\min\{1,\frac{n-2\gamma}{4}\}}\\
&\hspace{8mm}
+C\|u_\nu-\bar{u}_{z_\nu}\|_{L^{\frac{n+2\gamma}{n-2\gamma}}(M)}^{\frac{n+2\gamma}{n-2\gamma}}
+C\sum_{k=1}^m\varepsilon_{k,\nu}^{\frac{n-2\gamma}{2}}+o(1) \|u_\nu-\bar{u}_{z_\nu}\|_{L^1(M)}.
\end{split}
\end{equation*}
Since $\max\{0,1-\frac{n-2\gamma}{4}\}<1$, this implies
\begin{equation*}
\begin{split}
\|u_\nu-\bar{u}_{z_\nu}\|_{L^1(M)}
& \leq   C\Big\|(R_\gamma^{g_\nu}-s_\infty)u_\nu^{\frac{n+2\gamma}{n-2\gamma}}\Big\|_{L^{\frac{2n}{n+2\gamma}}(M)}\\
&\hspace{4mm}+
C\|u_\nu-\bar{u}_{z_\nu}\|_{L^{\frac{n+2\gamma}{n-2\gamma}}(M)}^{\frac{n+2\gamma}{n-2\gamma}}
+C\sum_{k=1}^m\varepsilon_{k,\nu}^{\frac{n-2\gamma}{2}}.
\end{split}
\end{equation*}
The assertion follows now from Lemma \ref{lem6.11}.
\end{proof}

\begin{lem}\label{lem6.13}
We have
\begin{equation*}
\begin{split}
&\sup_{a\in A}\left|\int_M\psi_a\Big(P_\gamma^{g_0}\bar{u}_{z_\nu}-s_\infty \bar{u}_{z_\nu}^{\frac{n+2\gamma}{n-2\gamma}}\Big)d\mu_{g_0}\right|\\
&\leq C\left(\int_M|R_\gamma^{g_\nu}-s_\infty|^{\frac{2n}{n+2\gamma}}u_\nu^{\frac{2n}{n-2\gamma}}d\mu_{g_0}\right)^{\frac{n+2\gamma}{2n}}
+C\sum_{k=1}^m\varepsilon_{k,\nu}^{\frac{n-2\gamma}{2}}
\end{split}
\end{equation*}
if $\nu$ is sufficiently large.
\end{lem}
\begin{proof}
Using $P_{g_0}^\gamma \psi_a=\lambda_a u_\infty^{\frac{4\gamma}{n-2\gamma}}\psi_a$, there holds 
\begin{equation*}
\begin{split}
&\int_M\psi_a\Big(P_\gamma^{g_0}\bar{u}_{z_\nu}-s_\infty \bar{u}_{z_\nu}^{\frac{n+2\gamma}{n-2\gamma}}\Big)d\mu_{g_0}\\
&=\int_M\psi_a\Big(P_\gamma^{g_0}u_\nu-s_\infty u_\nu^{\frac{n+2\gamma}{n-2\gamma}}\Big)d\mu_{g_0}
\\
&\hspace{4mm}-\lambda_a\int_Mu_\infty^{\frac{4\gamma}{n-2\gamma}}\psi_a(u_\nu-\bar{u}_{z_\nu})d\mu_{g_0}
+s_\infty\int_M\psi_a(u_\nu^{\frac{n+2\gamma}{n-2\gamma}}-\bar{u}_{z_\nu}^{\frac{n+2\gamma}{n-2\gamma}})d\mu_{g_0}.
\end{split}
\end{equation*}
Using the identity
$$P_\gamma^{g_0}u_\nu-s_\infty u_\nu^{\frac{n+2\gamma}{n-2\gamma}}=(R_\gamma^{g_\nu}-s_\infty) u_\nu^{\frac{n+2\gamma}{n-2\gamma}},$$
we obtain
\begin{equation*}
\begin{split}
&\int_M\psi_a\Big(P_\gamma^{g_0}\bar{u}_{z_\nu}-s_\infty \bar{u}_{z_\nu}^{\frac{n+2\gamma}{n-2\gamma}}\Big)d\mu_{g_0}\\
&=\int_M\psi_a u_\nu^{\frac{n+2\gamma}{n-2\gamma}}(R_\gamma^{g_\nu}-s_\infty) d\mu_{g_0}
\\
&\hspace{4mm}-\lambda_a\int_Mu_\infty^{\frac{4\gamma}{n-2\gamma}}\psi_a(u_\nu-\bar{u}_{z_\nu})d\mu_{g_0}
+s_\infty\int_M\psi_a(u_\nu^{\frac{n+2\gamma}{n-2\gamma}}-\bar{u}_{z_\nu}^{\frac{n+2\gamma}{n-2\gamma}})d\mu_{g_0}.
\end{split}
\end{equation*}
Using the pointwise estimate \eqref{eq: pointwise estimate 1}
\begin{equation*}
|u_\nu^{\frac{n+2\gamma}{n-2\gamma}}-\bar{u}_{z_\nu}^{\frac{n+2\gamma}{n-2\gamma}}|
\leq C\bar{u}_{z_\nu}^{\frac{4\gamma}{n-2\gamma}}|u_\nu-\bar{u}_{z_\nu}|+C|u_\nu-\bar{u}_{z_\nu}|^{\frac{n+2\gamma}{n-2\gamma}},
\end{equation*}
we conclude that
\begin{equation*}
\begin{split}
&\sup_{a\in A}\left|\int_M\psi_a\Big(P_\gamma^{g_0}\bar{u}_{z_\nu}-s_\infty \bar{u}_{z_\nu}^{\frac{n+2\gamma}{n-2\gamma}}\Big)d\mu_{g_0}\right|\\
&\hspace{8mm}\leq C\|u_\nu^{\frac{n+2\gamma}{n-2\gamma}}(R_\gamma^{g_\nu}-s_\infty)\|_{L^{\frac{2n}{n+2\gamma}}(M)}
\\
&\hspace{12mm}+C\|u_\nu-\bar{u}_{z_\nu}\|_{L^1(M)}+C\|u_\nu-\bar{u}_{z_\nu}\|_{L^{\frac{n+2\gamma}{n-2\gamma}}(M)}^{\frac{n+2\gamma}{n-2\gamma}}.
\end{split}
\end{equation*}
Now the assertion follows from combining this with Lemma \ref{lem6.11} and Lemma \ref{lem6.12}.
\end{proof}

\begin{prop}\label{prop6.14}
There holds
\begin{equation*}
\begin{split}
&E(\bar{u}_{z_\nu})-E(u_\infty)\\
&\hspace{4mm}\leq  C\left(\int_M|R_\gamma^{g_\nu}-s_\infty|^{\frac{2n}{n+2\gamma}}u_\nu^{\frac{2n}{n-2\gamma}}d\mu_{g_0}\right)^{\frac{n+2\gamma}{2n}(1+\delta)}
+C\sum_{k=1}^m\varepsilon_{k,\nu}^{\frac{n-2\gamma}{2}(1+\delta)}
\end{split}
\end{equation*}
if $\nu$ is sufficiently large.
\end{prop}
\begin{proof}
This follows immediately from Lemma \ref{lem6.5} and Lemma \ref{lem6.13}.
\end{proof}

\begin{prop}\label{prop6.15}
If $\nu$ is sufficiently large, then
\begin{equation*}
\begin{split}
E(v_\nu)&\leq \Big(E(\bar{u}_{z_\nu})^{\frac{n}{2\gamma}}+\sum_{k=1}^mE(u_{(x_{k,\nu},\varepsilon_{k,\nu})})^{\frac{n}{2\gamma}}\Big)^{\frac{2\gamma}{n}}
-c\sum_{k=1}^m\varepsilon_{k,\nu}^{\frac{n-2\gamma}{2}}.
\end{split}
\end{equation*}
\end{prop}
\begin{proof}
From the identity
\begin{equation*}
\begin{split}
\int_Mv_\nu P_\gamma^{g_0}v_\nu d\mu_{g_0}
&=\int_M\bar{u}_{z_\nu} P_\gamma^{g_0}\bar{u}_{z_\nu} d\mu_{g_0}+\int_M\sum_{k=1}^m\alpha_{k,\nu}^2u_{(x_{k,\nu},\varepsilon_{k,\nu})} P_\gamma^{g_0}(u_{(x_{k,\nu},\varepsilon_{k,\nu})})  d\mu_{g_0}\\
&\hspace{4mm}+2\int_M\sum_{k=1}^m\alpha_{k,\nu} \bar{u}_{z_\nu}P_\gamma^{g_0}(u_{(x_{k,\nu},\varepsilon_{k,\nu})})  d\mu_{g_0}\\
&\hspace{4mm}+2\int_M\sum_{i<j}\alpha_{i,\nu}\alpha_{j,\nu}u_{(x_{i,\nu},\varepsilon_{i,\nu})} P_\gamma^{g_0}(u_{(x_{j,\nu},\varepsilon_{j,\nu})})  d\mu_{g_0},
\end{split}
\end{equation*}
we rewrite to
\begin{equation*}
\begin{split}
E(v_\nu)&\left(\int_Mv_\nu^{\frac{2n}{n-2\gamma}} d\mu_{g_0}\right)^{\frac{n-2\gamma}{n}}\\
&=\int_M F(\bar{u}_{z_\nu})\bar{u}_{z_\nu}^{\frac{2n}{n-2\gamma}} d\mu_{g_0}+\int_M\sum_{k=1}^m\alpha_{k,\nu}^2
F(u_{(x_{k,\nu},\varepsilon_{k,\nu})}) u_{(x_{k,\nu},\varepsilon_{k,\nu})}^{\frac{2n}{n-2\gamma}} d\mu_{g_0}\\
&\hspace{4mm}+2\int_M\sum_{k=1}^m\alpha_{k,\nu} \bar{u}_{z_\nu}P_\gamma^{g_0}(u_{(x_{k,\nu},\varepsilon_{k,\nu})})  d\mu_{g_0}\\
&\hspace{4mm}+2\int_M\sum_{i<j}\alpha_{i,\nu}\alpha_{j,\nu}u_{(x_{i,\nu},\varepsilon_{i,\nu})} P_\gamma^{g_0}(u_{(x_{j,\nu},\varepsilon_{j,\nu})})  d\mu_{g_0},
\end{split}
\end{equation*}
where $F$ is defined as in (\ref{48}). Moreover, we have 
\begin{equation*}
\begin{split}
&\Big(E(\bar{u}_{z_\nu})^{\frac{n}{2\gamma}}+\sum_{k=1}^mE(u_{(x_{k,\nu},\varepsilon_{k,\nu})})^{\frac{n}{2\gamma}}\Big)^{\frac{2\gamma}{n}}\left(\int_Mv_\nu^{\frac{2n}{n-2\gamma}} d\mu_{g_0}\right)^{\frac{n-2\gamma}{n}}\\
&=\left(\int_M\Big(F(\bar{u}_{z_\nu})^{\frac{n}{2\gamma}}\bar{u}_{z_\nu}^{\frac{2n}{n-2\gamma}}
+\sum_{k=1}^mF(u_{(x_{k,\nu},\varepsilon_{k,\nu})})^{\frac{n}{2\gamma}}u_{(x_{k,\nu},\varepsilon_{k,\nu})}^{\frac{2n}{n-2\gamma}}\Big)d\mu_{g_0}\right)^{\frac{2\gamma}{n}}\\
&\hspace{8mm}\cdot\left(\int_Mv_\nu^{\frac{2n}{n-2\gamma}} d\mu_{g_0}\right)^{\frac{n-2\gamma}{n}}\\
&\geq \int_M\Big(F(\bar{u}_{z_\nu})^{\frac{n}{2\gamma}}\bar{u}_{z_\nu}^{\frac{2n}{n-2\gamma}}
+\sum_{k=1}^mF(u_{(x_{k,\nu},\varepsilon_{k,\nu})})^{\frac{n}{2\gamma}}u_{(x_{k,\nu},\varepsilon_{k,\nu})}^{\frac{2n}{n-2\gamma}}\Big)^{\frac{2\gamma}{n}}v_\nu^2d\mu_{g_0}\\
&\geq \int_MF(\bar{u}_{z_\nu}) \bar{u}_{z_\nu}^{\frac{2n}{n-2\gamma}}d\mu_{g_0}
+\int_M\sum_{k=1}^m\alpha_{k,\nu}^2F(u_{(x_{k,\nu},\varepsilon_{k,\nu})})u_{(x_{k,\nu},\varepsilon_{k,\nu})}^{\frac{2n}{n-2\gamma}}d\mu_{g_0}\\
&\hspace{4mm}+2\int_M\sum_{l=1}^m\alpha_{l,\nu}\Big(F(\bar{u}_{z_\nu})^{\frac{n}{2\gamma}}\bar{u}_{z_\nu}^{\frac{2n}{n-2\gamma}}
+\sum_{k=1}^mF(u_{(x_{k,\nu},\varepsilon_{k,\nu})})^{\frac{n}{2\gamma}}u_{(x_{k,\nu},\varepsilon_{k,\nu})}^{\frac{2n}{n-2\gamma}}\Big)^{\frac{2\gamma}{n}}\\
&\hspace{12mm}\cdot
\bar{u}_{z_\nu}u_{(x_{k,\nu},\varepsilon_{k,\nu})} d\mu_{g_0}\\
&\hspace{4mm}+2\int_M\sum_{i<j}\alpha_{i,\nu}\alpha_{j,\nu}\Big(F(u_{(x_{i,\nu},\varepsilon_{i,\nu})})^{\frac{n}{2\gamma}}u_{(x_{i,\nu},\varepsilon_{i,\nu})}^{\frac{2n}{n-2\gamma}}+F(u_{(x_{j,\nu},\varepsilon_{j,\nu})})^{\frac{n}{2\gamma}}u_{(x_{j,\nu},\varepsilon_{j,\nu})}^{\frac{2n}{n-2\gamma}}\Big)^{\frac{2\gamma}{n}}\\
&\hspace{12mm}\cdot
u_{(x_{i,\nu},\varepsilon_{i,\nu})}u_{(x_{j,\nu},\varepsilon_{j,\nu})} d\mu_{g_0}
\end{split}
\end{equation*}
by H\"{o}lder's inequality. Note that in the last inequality we used the positivity of each terms so that
\begin{align*}
    &\left(F(\bar{u}_{z_\nu})^{\frac{n}{2\gamma}}\bar{u}_{z_\nu}^{\frac{2n}{n-2\gamma}}+\sum_{k=1}^nF(u_{(x_{k,\nu},\varepsilon_{k,\nu})})^{\frac{n}{2\gamma}}u_{(x_{k,\nu},\varepsilon_{k,\nu})}^{\frac{2n}{n-2\gamma}})\right)^{\frac{2\gamma}{n}}\\
    &\geq\Big(F(u_{(x_{i,\nu},\varepsilon_{i,\nu})})^{\frac{n}{2\gamma}}u_{(x_{i,\nu},\varepsilon_{i,\nu})}^{\frac{2n}{n-2\gamma}}+F(u_{(x_{j,\nu},\varepsilon_{j,\nu})})^{\frac{n}{2\gamma}}u_{(x_{j,\nu},\varepsilon_{j,\nu})}^{\frac{2n}{n-2\gamma}}\Big)^{\frac{2\gamma}{n}}.
\end{align*}
Let $1\leq k\leq m$.
From the positivity of each terms, we have the identity
\begin{equation*}
\begin{split}
&\Big(F(\bar{u}_{z_\nu})^{\frac{n}{2\gamma}}\bar{u}_{z_\nu}^{\frac{2n}{n-2\gamma}}
+F(u_{(x_{k,\nu},\varepsilon_{k,\nu})})^{\frac{n}{2\gamma}}u_{(x_{k,\nu},\varepsilon_{k,\nu})}^{\frac{2n}{n-2\gamma}}\Big)^{\frac{2\gamma}{n}}\bar{u}_{z_\nu}u_{(x_{k,\nu},\varepsilon_{k,\nu})} \\
&\hspace{4mm}\geq F(\bar{u}_{z_\nu})\bar{u}_{z_\nu}^{\frac{n+2\gamma}{n-2\gamma}}u_{(x_{k,\nu},\varepsilon_{k,\nu})}
+c\,\varepsilon_{k,\nu}^{-\frac{n+2\gamma}{2}}1_{\{d(x_{k,\nu},x)\leq\varepsilon_{k,\nu}\}}.
\end{split}
\end{equation*}
Integrate, we obtain
\begin{equation*}
\begin{split}
&\int_M \Big(F(\bar{u}_{z_\nu})^{\frac{n}{2\gamma}}\bar{u}_{z_\nu}^{\frac{2n}{n-2\gamma}}
+F(u_{(x_{k,\nu},\varepsilon_{k,\nu})})^{\frac{n}{2\gamma}}u_{(x_{k,\nu},\varepsilon_{k,\nu})}^{\frac{2n}{n-2\gamma}}\Big)^{\frac{2\gamma}{n}}\bar{u}_{z_\nu}u_{(x_{k,\nu},\varepsilon_{k,\nu})}
d\mu_{g_0}\\
&\geq \int_MF(\bar{u}_{z_\nu})\bar{u}_{z_\nu}^{\frac{n+2\gamma}{n-2\gamma}}u_{(x_{k,\nu},\varepsilon_{k,\nu})}d\mu_{g_0}
+c\,\varepsilon_{k,\nu}^{\frac{n-2\gamma}{2}}
\end{split}
\end{equation*}
if $\nu$ is sufficiently large. We next consider a pair $i<j$. Note that (\ref{4.13})  still hold.
Integration (\ref{4.13}) over $M$ yields
\begin{equation*}
\begin{split}
&\int_M \Big(F(u_{(x_{i,\nu},\varepsilon_{i,\nu})})^{\frac{n}{2\gamma}}u_{(x_{i,\nu},\varepsilon_{i,\nu})}^{\frac{2n}{n-2\gamma}}
+F(u_{(x_{j,\nu},\varepsilon_{j,\nu})})^{\frac{n}{2\gamma}}u_{(x_{j,\nu},\varepsilon_{j,\nu})}^{\frac{2n}{n-2\gamma}}\Big)^{\frac{2\gamma}{n}}u_{(x_{i,\nu},\varepsilon_{i,\nu})}
u_{(x_{j,\nu},\varepsilon_{j,\nu})}
d\mu_{g_0}\\
&\geq \int_MF(u_{(x_{j,\nu},\varepsilon_{j,\nu})})u_{(x_{i,\nu},\varepsilon_{i,\nu})}u_{(x_{j,\nu},\varepsilon_{j,\nu})}^{\frac{n+2\gamma}{n-2\gamma}}d\mu_{g_0}
+c\Big(\frac{\varepsilon_{j,\nu}^2+d(x_{i,\nu},x_{j,\nu})^2}{\varepsilon_{i,\nu}\varepsilon_{j,\nu}}\Big)^{-\frac{n-2\gamma}{2}}
\end{split}
\end{equation*}
if $\nu$ is sufficiently large. From these, it follows that
\begin{equation*}
\begin{split}
&\Big(E(\bar{u}_{z_\nu})^{\frac{n}{2\gamma}}+\sum_{k=1}^mE(u_{(x_{k,\nu},\varepsilon_{k,\nu})})^{\frac{n}{2\gamma}}\Big)^{\frac{2\gamma}{n}}\left(\int_Mv_\nu^{\frac{2n}{n-2\gamma}} d\mu_{g_0}\right)^{\frac{n-2\gamma}{n}}\\
&\geq  \int_MF(\bar{u}_{z_\nu}) \bar{u}_{z_\nu}^{\frac{2n}{n-2\gamma}}d\mu_{g_0}
+\int_M\sum_{k=1}^m\alpha_{k,\nu}^2F(u_{(x_{k,\nu},\varepsilon_{k,\nu})})u_{(x_{k,\nu},\varepsilon_{k,\nu})}^{\frac{2n}{n-2\gamma}}d\mu_{g_0}\\
&\hspace{4mm}+2\int_M\sum_{k=1}^m\alpha_{k,\nu}
F(\bar{u}_{z_\nu})\bar{u}_{z_\nu}^{\frac{n+2\gamma}{n-2\gamma}}u_{(x_{k,\nu},\varepsilon_{k,\nu})}d\mu_{g_0}
\\
&\hspace{4mm}+2\int_M\sum_{i<j}\alpha_{i,\nu}\alpha_{j,\nu}F(u_{(x_{j,\nu},\varepsilon_{j,\nu})})u_{(x_{i,\nu},\varepsilon_{i,\nu})}u_{(x_{j,\nu},\varepsilon_{j,\nu})}^{\frac{n+2\gamma}{n-2\gamma}}d\mu_{g_0}
\\
&\hspace{4mm}+c\sum_{k=1}^m\varepsilon_{k,\nu}^{\frac{n-2\gamma}{2}}+c\sum_{i<j}\Big(\frac{\varepsilon_{j,\nu}^2+d(x_{i,\nu},x_{j,\nu})^2}{\varepsilon_{i,\nu}\varepsilon_{j,\nu}}\Big)^{-\frac{n-2\gamma}{2}}.
\end{split}
\end{equation*}
Putting these facts together, we obtain
\begin{equation*}
\begin{split}
&E(v_\nu)\left(\int_Mv_\nu^{\frac{2n}{n-2\gamma}} d\mu_{g_0}\right)^{\frac{n-2\gamma}{n}}\\
&\leq \Big(E(\bar{u}_{z_\nu})^{\frac{n}{2\gamma}}+\sum_{k=1}^mE(u_{(x_{k,\nu},\varepsilon_{k,\nu})})^{\frac{n}{2\gamma}}\Big)^{\frac{2\gamma}{n}}\left(\int_Mv_\nu^{\frac{2n}{n-2\gamma}}d\mu_{g_0}\right)^{\frac{n-2\gamma}{n}}\\
&\hspace{4mm}+2\int_M\sum_{k=1}^m\alpha_{k,\nu}
\Big(P_\gamma^{g_0}\bar{u}_{z_\nu}-F(\bar{u}_{z_\nu})\bar{u}_{z_\nu}^{\frac{n+2\gamma}{n-2\gamma}}\Big)u_{(x_{k,\nu},\varepsilon_{k,\nu})}d\mu_{g_0}\\
&\hspace{4mm}+2\int_M\sum_{i<j}\alpha_{i,\nu}\alpha_{j,\nu}
u_{(x_{i,\nu},\varepsilon_{i,\nu})}
\Big(P_\gamma^{g_0}u_{(x_{j,\nu},\varepsilon_{j,\nu})}-F(u_{(x_{j,\nu},\varepsilon_{j,\nu})})u_{(x_{j,\nu},\varepsilon_{j,\nu})}^{\frac{n+2\gamma}{n-2\gamma}}\Big)
d\mu_{g_0}\\
&\hspace{4mm}-c\sum_{k=1}^m\varepsilon_{k,\nu}^{\frac{n-2\gamma}{2}}-c\sum_{i<j}\Big(\frac{\varepsilon_{j,\nu}^2+d(x_{i,\nu},x_{j,\nu})^2}{\varepsilon_{i,\nu}\varepsilon_{j,\nu}}\Big)^{-\frac{n-2\gamma}{2}}.
\end{split}
\end{equation*}
Note that
\begin{equation*}
\begin{split}
\int_M\Big|P_\gamma^{g_0}\bar{u}_{z_\nu}-F(\bar{u}_{z_\nu})\bar{u}_{z_\nu}^{\frac{n+2\gamma}{n-2\gamma}}\Big|u_{(x_{k,\nu},\varepsilon_{k,\nu})}d\mu_{g_0}
\leq C\int_M u_{(x_{k,\nu},\varepsilon_{k,\nu})}\,d\mu_{g_0}= o(1)\varepsilon_{k,\nu}^{\frac{n-2\gamma}{2}}.
\end{split}
\end{equation*}
With this and \eqref{eq: Main interaction of bubble term estimate}, the assertion follows.
\end{proof}

\begin{cor}\label{cor6.16}
If $\nu$ is sufficiently large, then
\begin{equation*}
\begin{split}
E(v_\nu)&\leq \Big(E(u_\infty)^{\frac{n}{2\gamma}}+mY_\gamma(\mathbb{S}^n)^{\frac{n}{2\gamma}}\Big)^{\frac{2\gamma}{n}}\\
&\hspace{4mm}+C\left(\int_M u_\nu^{\frac{2n}{n-2\gamma}}|R_\gamma^{g_\nu}-s_\infty|^{\frac{2n}{n+2\gamma}} d\mu_{g_0}\right)^{\frac{n+2\gamma}{2n}(1+\delta)}.
\end{split}
\end{equation*}
\end{cor}

\begin{proof}
    Using Proposition \ref{prop6.15}, Proposition \ref{prop6.14}, and the fact that $E(u_{(x_{k,\nu},\varepsilon_{k,\nu})})\leq Y_\gamma (\mathbb{S}^n)$ from \eqref{eq: energy for Schoen's bubble estimation}, the proof is complete.
\end{proof}

%% file: Proof_of_main_prop.tex
\section{Proof of Lemma \ref{prop1.1}}

From $R_\gamma^{g_\nu}=u^{-\frac{n+2\gamma}{n-2\gamma}}P_\gamma^{g_0}u$, we have
\begin{equation*}
\begin{split}
E(u_\nu)
&=\int_M u_\nu P_\gamma^{g_0} u_\nu d\mu_{g_0}\\
&=\int_M v_\nu P_\gamma^{g_0} v_\nu d\mu_{g_0}+2\int_M u_\nu^{\frac{n+2\gamma}{n-2\gamma}} R_\gamma^{g_\nu} w_\nu d\mu_{g_0}
-\int_M w_\nu P_\gamma^{g_0} w_\nu d\mu_{g_0}.
\end{split}
\end{equation*}
We rewrite to
\begin{equation}\label{172}
\begin{split}
s_\gamma^{g_\nu}
&=E(v_\nu)\left(\int_M v_\nu^{\frac{2n}{n-2\gamma}}d\mu_{g_0}\right)^{\frac{n-2\gamma}{n}}
+2\int_M u_\nu^{\frac{n+2\gamma}{n-2\gamma}} (R_\gamma^{g_\nu}-s_\infty) w_\nu d\mu_{g_0}\\
&\hspace{4mm}-\int_M w_\nu \Big(P_\gamma^{g_0} w_\nu-\frac{n+2\gamma}{n-2\gamma}s_\infty v_\nu^{\frac{4\gamma}{n-2\gamma}} w_\nu\Big) d\mu_{g_0}\\
&\hspace{4mm}+s_\infty\int_M\Big(-\frac{n+2\gamma}{n-2\gamma} v_\nu^{\frac{4\gamma}{n-2\gamma}} w_\nu^2+2(v_\nu+w_\nu)^{\frac{n+2\gamma}{n-2\gamma}} w_\nu\Big)d\mu_{g_0}.
\end{split}
\end{equation}
It follows from the normalization \eqref{1.5} that
\begin{equation}\label{173}
\int_M u_\nu^{\frac{2n}{n-2\gamma}}\,d\mu_{g_0}=\int_M(v_\nu+w_\nu)^{\frac{2n}{n-2\gamma}}d\mu_{g_0}=1.
\end{equation}
Since $x\mapsto x^{\frac{n-2\gamma}{n}}$ is a concave function, we have
\begin{equation*}
\begin{split}
\left(\int_Mv_\nu^{\frac{2n}{n-2\gamma}} d\mu_{g_0}\right)^{\frac{n-2\gamma}{n}}-1
&\leq\frac{n-2\gamma}{n}\left(\int_Mv_\nu^{\frac{2n}{n-2\gamma}} d\mu_{g_0}-1\right)\\
&=\int_M\Big(\frac{n-2\gamma}{n}v_\nu^{\frac{2n}{n-2\gamma}} -\frac{n-2\gamma}{n}(v_\nu+w_\nu)^{\frac{2n}{n-2\gamma}}\Big) d\mu_{g_0}
\end{split}
\end{equation*}
by (\ref{173}). Combining this with
(\ref{172}), we obtain
\begin{equation}\label{176}
\begin{split}
s_\gamma^{g_\nu}
&\leq s_\infty+(E(v_\nu)-s_\infty)\left(\int_M v_\nu^{\frac{2n}{n-2\gamma}}d\mu_{g_0}\right)^{\frac{n-2\gamma}{n}}\\
&\hspace{4mm}+2\int_M u_\nu^{\frac{n+2\gamma}{n-2\gamma}} (R_\gamma^{g_\nu}-s_\infty) w_\nu d\mu_{g_0}\\
&\hspace{4mm}-\int_M w_\nu \Big(P_\gamma^{g_0} w_\nu-\frac{n+2\gamma}{n-2\gamma}s_\infty v_\nu^{\frac{4\gamma}{n-2\gamma}} w_\nu\Big) d\mu_{g_0}\\
&\hspace{4mm}+s_\infty\int_M\Big(
\frac{n-2\gamma}{n}v_\nu^{\frac{2n}{n-2\gamma}}-\frac{n-2\gamma}{n}(v_\nu+w_\nu)^{\frac{2n}{n-2\gamma}}\\
&\hspace{28mm}-\frac{n+2\gamma}{n-2\gamma} v_\nu^{\frac{4\gamma}{n-2\gamma}} w_\nu^2+2(v_\nu+w_\nu)^{\frac{n+2\gamma}{n-2\gamma}} w_\nu\Big)d\mu_{g_0}.
\end{split}
\end{equation}
It follows from H\"{o}lder's inequality that
\begin{equation}\label{177}
\begin{split}
&\int_M u_\nu^{\frac{n+2\gamma}{n-2\gamma}}(R_\gamma^{g_\nu}-s_\infty) w_\nu d\mu_{g_0}\\
&\leq\left(\int_M u_\nu^{\frac{2n}{n-2\gamma}}|R_\gamma^{g_\nu}-s_\infty|^{\frac{2n}{n+2\gamma}} d\mu_{g_0}\right)^{\frac{n+2\gamma}{2n}}
\left(\int_M |w_\nu|^{\frac{2n}{n-2\gamma}} d\mu_{g_0}\right)^{\frac{n-2\gamma}{2n}}.
\end{split}
\end{equation}
Moreover, it follows from Corollary \ref{cor3.5} and Corollary \ref{cor6.10} that
\begin{equation}\label{178}
\begin{split}
&\int_M\Big(w_\nu P_\gamma^{g_0} w_\nu-\frac{n+2\gamma}{n-2\gamma} s_\infty v_\nu^{\frac{4\gamma}{n-2\gamma}}w_\nu^2\Big)d\mu_{g_0}\\
&\geq c\int_M w_\nu P_\gamma^{g_0} w_\nu d\mu_{g_0}\geq c\left(\int_M|w_\nu|^{\frac{2n}{n-2\gamma}}d\mu_{g_0}\right)^{\frac{n-2\gamma}{n}}.
\end{split}
\end{equation}
Moreover, it follows from the pointwise estimate \eqref{eq: pointwise estimate 3}
\begin{equation*}
\begin{split}
&\Big|
\frac{n-2\gamma}{n}v_\nu^{\frac{2n}{n-2\gamma}} -\frac{n-2\gamma}{n}(v_\nu+w_\nu)^{\frac{2n}{n-2\gamma}}-\frac{n+2\gamma}{n-2\gamma} v_\nu^{\frac{4\gamma}{n-2\gamma}} w_\nu^2+2(v_\nu+w_\nu)^{\frac{n+2\gamma}{n-2\gamma}} w_\nu\Big|\\
&\leq Cv_\nu^{\max\{0,\frac{4\gamma}{n-2\gamma}-1\}}|w_\nu|^{\min\{\frac{2n}{n-2\gamma},3\}}+C|w_\nu|^{\frac{2n}{n-2\gamma}}
\end{split}
\end{equation*}
that
\begin{equation*}
\begin{split}
&\int_M\Big(
\frac{n-2\gamma}{n}v_\nu^{\frac{2n}{n-2\gamma}} -\frac{n-2\gamma}{n}(v_\nu+w_\nu)^{\frac{2n}{n-2\gamma}}-\frac{n+2\gamma}{n-2\gamma} v_\nu^{\frac{4\gamma}{n-2\gamma}} w_\nu^2+2(v_\nu+w_\nu)^{\frac{n+2\gamma}{n-2\gamma}} w_\nu\Big)d\mu_{g_0}\\
&\leq C\int_Mv_\nu^{\max\{0,\frac{4\gamma}{n-2\gamma}-1\}}|w_\nu|^{\min\{\frac{2n}{n-2\gamma},3\}}d\mu_{g_0}+C\int_M|w_\nu|^{\frac{2n}{n-2\gamma}}d\mu_{g_0}\\
&\leq C\left(\int_M|w_\nu|^{\frac{2n}{n-2\gamma}}d\mu_{g_0}\right)^{\frac{n-2\gamma}{2n}\min\{\frac{2n}{n-2\gamma},3\}}
\end{split}
\end{equation*}
This together with (\ref{176})-(\ref{178}) implies that
\begin{equation*}
\begin{split}
s_\gamma^{g_\nu}
&\leq s_\infty+(E(v_\nu)-s_\infty)\left(\int_M v_\nu^{\frac{2n}{n-2\gamma}}d\mu_{g_0}\right)^{\frac{n-2\gamma}{n}}\\
&+C\left(\int_M u_\nu^{\frac{2n}{n-2\gamma}}|R_\gamma^{g_\nu}-s_\infty|^{\frac{2n}{n+2\gamma}} d\mu_{g_0}\right)^{\frac{n+2\gamma}{2n}}
\end{split}
\end{equation*}
From (\ref{2.7}), Corollary \ref{cor5.7}, and Corollary \ref{cor6.16}, we obtain
\[
E(v_\nu)-s_\infty\leq C\left(\int_M u_\nu^{\frac{2n}{n-2\gamma}}|R_\gamma^{g_\nu}-s_\infty|^{\frac{2n}{n+2\gamma}} d\mu_{g_0}\right)^{\frac{n+2\gamma}{2n}(1+\delta)}.
\]
Therefore, we conclude
\begin{equation*}
s_\gamma^{g_\nu}\leq s_\infty
+C\left(\int_M u_\nu^{\frac{2n}{n-2\gamma}}|R_\gamma^{g_\nu}-s_\infty|^{\frac{2n}{n+2\gamma}} d\mu_{g_0}\right)^{\frac{n+2\gamma}{2n}(1+\delta)}.
\end{equation*}
This completes the proof of Proposition
\ref{prop1.1}.

\appendix

%% file: Appendix.tex
\appendix

\section{Axillary estimates on differences with exponents}

\begin{lem}
    For any $a,b\geq 0$ and $p>1$,
    \begin{equation}\label{eq: pointwise estimate 1}
        |a^p-b^p|\leq C(|a-b|^p+b^{p-1}|a-b|),
    \end{equation}
    \begin{equation}\label{eq: pointwise estimate 2}
        |a^p-b^p-pb^{p-1}(a-b)|\leq C(b^{\max\{0,p-2\}}|a-b|^{\min\{p,2\}}+|a-b|^p),
    \end{equation}
    and
    \begin{equation}\label{eq: pointwise estimate 3}
    \begin{split}
        &\left|a^{p}-b^{p}-p\,b^{p-1}(a-b)
               -\frac12 p(p-1)b^{p-2}(a-b)^{2}\right|\\
               &\leq C\left(b^{\max\{0,p-3\}}|a-b|^{\min\{p,3\}}+|a-b|^{p}\right),
    \end{split}
    \end{equation}
    where the constant $C$ depends only on $p$.
\end{lem}

\begin{proof}
    Since
    \[
        |a^{p}-b^{p}|\le C(a^{p-1}+b^{p-1})|a-b|
    \]
    and
    \[
        a^{p-1}=(b+(a-b))^{p-1}\le C(b^{p-1}+|a-b|^{p-1}),
    \]
    \eqref{eq: pointwise estimate 1} follows immediately.

    We show \eqref{eq: pointwise estimate 2}. Set $\delta:=a-b$ and expand $f(s)=s^{p}$ to second order at~$b$:
    \[
        a^{p}=b^{p}+p\,b^{p-1}\delta+\frac12 p(p-1)\xi^{p-2}\delta^{2},
    \]
    for some $\xi\in[\min\{a,b\},\max\{a,b\}]$.
    Hence
    \[
        |a^{p}-b^{p}-p\,b^{p-1}\delta|
        \;=\;
        C\,\xi^{\,p-2}|\delta|^{2},
    \]
    with $C=\frac12 p(p-1)$.
    
    Now, suppose $p\geq 2$. Because $p-2\ge0$ and $\xi$ lies between $a$ and $b$,
    \[
        \xi^{p-2}\le(a+b)^{p-2}\le C(a^{p-2}+b^{p-2}),
    \]
    and
    \[
        a^{p-2}=(b+(a-b))^{p-2}\leq C(b^{p-2}+|\delta|^{p-2}),
    \]
    whence
    \[
        |a^{p}-b^{p}-p\,b^{p-1}\delta|
        \le C\bigl(b^{p-2}|\delta|^{2}+|\delta|^{p}\bigr).
    \]
    On the other hand, if $1<p<2$, then $s^{p-2}$ is decreasing. Consequently $\xi^{p-2}\le b^{p-2}$ and, as in the usual dichotomy
    $|\delta|\lessgtr b$, one shows $b^{p-2}|\delta|^{2}\le|\delta|^{p}$.
    Thus the same bound as above holds, which is precisely
    \eqref{eq: pointwise estimate 2}.

    \eqref{eq: pointwise estimate 3} follows similarly to \eqref{eq: pointwise estimate 2}. This completes the proof.
\end{proof}

\section{Construction of test bubbles}\label{Appendix B}


In this appendix, we introduce the test bubbles constructed in \cite{mayer2024fractional}, and moreover provide the proof of $L^{\frac{2n}{n+2\gamma}}$ estimate on the error of the bubble \ref{eq: L 2n / n+2gamma estimate on bubble}. In the following, we use notations $B_r^g(x)$ and $B_r^{g_+,+}(x)$ for geodesic balls on $M$ and $X$, centered at $x$, with respect to the metric $g$ and $g_+$, respectively.

For the construction of the bubbles, we are following \cite{mayer2024fractional}. Using the existence of conformal normal coordinates, there exists for every $x\in M$ a conformal factor $u_x\in C^\infty(M)$ such that
\[
    u_x>0,\quad \frac{1}{C}\leq u_x\leq C,\quad u_x(x)=1,\quad\text{and}\quad \nabla u_x(x)=0,
\]
inducing a conformal normal coordinate system close to $x$ on $M$, in particular in normal coordinates with respect to 
\[
    g_x:=u_x^{\frac{4}{n-2\gamma}}g_0\quad\text{with}\quad d\mu_{g_x}=u_x^{\frac{2n}{n-2\gamma}}d\mu_{g_0},
\]
we have for small $\varepsilon>0$, that $g_x=\delta+O(|x|^2)$, $\det g_x\equiv 1$ on $B_\varepsilon^{g_x}(x)$. 
As clarified in \cite[Subsection 3.2]{mayer2021asymptotics}, the conformal factor $u_x$ then naturally extends onto $X$ via
\[
    u_x=\left(\frac{y_x}{y}\right)^{\frac{n-2\gamma}{2}},
\]
where $y_x$ close to the boundary $M$ is the unique geodesic defining function, for which
\[
    g_x=y_x^2g_+,\quad g_x=dy_x^2+g_{x,y_x}\quad\text{near}\quad M\quad\text{with}\quad g_{x,y_x}|_M=g_x
\]
and there still holds $H_{g_x}=0$. Consequently
\[
    g_x=\delta+O(y+|x|^2)\quad\text{and}\quad\det g_x=1+O(y^2)\quad\text{in}\quad B_\varepsilon^{g_x,+}(x).
\]
We introduce an elliptic operator on $X$
\[
    D_{g_x}U=-\text{div}_{g_x}(y^{1-2\gamma}\nabla_{g_x}U)+E_{g_x}U,
\]
where
\[
    E_{g_x}:=y^{\frac{1-2\gamma}{2}}L_{g_x}y^{\frac{1-2\gamma}{2}}-\left(\frac{R_{g_+}}{c_n}+s(n-s)\right)y^{(1-2\gamma)-2},
\]
with
\[
    L_{g_x}=\Delta_{g_x}+\frac{R_g}{c_n},\quad c_n=\frac{4n}{n-1},\quad s=\frac{n}{2}+\gamma.
\]
It can be shown that (see \cite[(1-6)]{KimMussoWei})
\begin{equation}\label{eq: Egx approximation}
    E_{g_x}=\frac{n-2\gamma}{4n}(R_{g_x}-(n(n+1)+R_{g_+})y^{-2})y^{1-2\gamma}\quad\text{near}\quad M.
\end{equation}
Moreover, from $H\equiv 0$, we have
\begin{equation}\label{eq: decying order in Egx from zero mean curvature}
    n(n+1)+R_{g_+}=O(y),
\end{equation}
see \cite[(2-3) and Lemma 2.3]{KimMussoWei}.

With respect to $g_x$-normal coordinates $\mathbf{x}=\mathbf{x}_x$ centered at $x$ we define the standard bubble
    \[
        \delta_{x,\varepsilon}:=\left(\frac{\varepsilon}{\varepsilon^2+|\mathbf{x}|^2}\right)^{\frac{n-2\gamma}{2}}\quad\text{on}\quad B_{4\rho_0}^{g_x}(x)
    \]
for some small $\rho_0>0$, identifying $M\ni x\sim 0\in\mathbb{R}^n$. Note that on $\mathbb{R}^n$ with usual Euclidean metric, the standard bubble satisfies
    \[
        (-\Delta_{\mathbb{R}^n})^\gamma \delta_{x,\varepsilon}=c_{n,\gamma}^*\delta_{x,\varepsilon}^{\frac{n+2\gamma}{n-2\gamma}},\quad\text{on}\quad\mathbb{R}^n.
    \]
with positive constant $c_{n,\gamma}^*>0$ only depending on $n,\gamma$. And the Schoen's bubble associated to $\delta_{x,\varepsilon}$ is defined by
    \[
        \varphi_{x,\varepsilon,\delta}=\eta_{x,\delta}\delta_{x,\varepsilon}+(1-\eta_{x,\varepsilon})\varepsilon^{\frac{n-2\gamma}{2}}\frac{G_{g_x}(\cdot,x)}{c_{n,\gamma}}\quad\text{on}\quad M,
    \]
where $G_{g_x}$ is the Green's function associated with the metric $g_x$, $c_{n,\gamma}>0$ is the constant in the Positive Mass conjecture, and $\eta_{x,\delta}$ is a cut-off function defined in $g_x$-normal Fermi-coordinates by
    \[
        \left\{
            \begin{alignedat}{2}
                \eta_{x,\delta}:=&\eta\left(\frac{y^2+|\mathbf{x}|^2}{\delta^2}\right),&\\
                \eta\equiv& 1&\quad&\text{on}\quad [0,1].\\
                \eta\equiv&0&\quad&\text{on}\quad[2,\infty).
            \end{alignedat}
        \right.
    \]
    In what follows we will always choose $\delta\sim \varepsilon^{1/k}$ with $1\ll k\in\mathbb{N}$, and hence may write $\varphi_{x,\varepsilon,\delta}=\varphi_{x,\varepsilon}$ by abusing the notation. We define our Projective bubbles on $X$ by
    \[
        \overline{\varphi_{x,\varepsilon}}(z):=K_{g_x}\ast \varphi_{x,\varepsilon}(z):=\int_{M}K_{g_x}(z,\xi) \varphi_{x,\varepsilon}(\xi) \,d\mu_{g_x}(\xi),
    \]
    where $K_g$ denotes the Poisson Kernel introduced in \cite{mayer2021asymptotics,mayer2024fractional}. Given this definition, $\overline{\varphi_{x,\varepsilon}}$ uniquely solves 
    \[
        \left\{
            \begin{alignedat}{2}
                D_{g_x}\overline{\varphi_{x,\varepsilon}}&=0&\quad\text{in}\quad X\\
                \overline{\varphi_{x,\varepsilon}}&=\varphi_{x,\varepsilon}&\quad\text{on}\quad M,
            \end{alignedat}
        \right.
    \]
    and is the canonical extension of $\varphi_{x,\varepsilon}$ to $X$ with respect to $D_{g_x}$.

    On the other hand, instead of patch-and-extend type of bubble $\overline{\varphi_{x,\varepsilon}}$ on $X$, one can consider the extend-and-patch type bubble. Using $g_x$-normal Fermi-Coordinates $z=z_x=(\mathbf{x}_x,y)=(\mathbf{x},y)$, we consider as a second extension
    \[
        \widetilde{\varphi_{x,\varepsilon}}:=\eta_{x,\delta}\hat{\delta}_{x,\varepsilon}+(1-\eta_{x,\delta})\varepsilon^{\frac{n-2\gamma}{2}}\frac{\Gamma_x}{c_{n,\gamma}},\quad \Gamma_x=\Gamma_{g_x}(\cdot,x),
    \]
    where $\Gamma_{g_x}$ is a extended Green's function, i.e.
    \[
        \Gamma_{g_x}(z,x)=(K_{g_x}\ast G_{g_x}(\cdot,x))(z)=\int_M K_{g_x}(z,\xi)G_{g_x}(\xi,x)\,d\mu_{g_x}(\xi),\quad z\in \overline{X},\,\xi\in M.
    \]
    and $\hat{\delta}_{x,\varepsilon}$ is the extended standard bubble, i.e.
    \[
        \left\{
            \begin{alignedat}{3}
                \text{div}(y^{1-2\gamma}\nabla \hat{\delta}_{x,\varepsilon})&=0&\quad&\text{in}\quad&&\mathbb{R}_+^{n+1}\\
                \hat{\delta}_{x,\varepsilon}&=\delta_{x,\varepsilon}&\quad&\text{on}\quad &&\mathbb{R}^n\\
                -c_\gamma\lim_{y\rightarrow 0}y^{1-2\gamma}\partial_y\hat{\delta}_{x,\varepsilon}&=c^*_{n,\gamma}\delta_{x,\varepsilon}^{}\frac{n+2\gamma}{n-2\gamma}&\quad&\text{on}\quad&&\mathbb{R}^n.
            \end{alignedat}
        \right.
    \]
    Under this construction, $\widetilde{\varphi_{x,\varepsilon}}$ enjoys weaker identities, namely
    \[
        \left\{
            \begin{alignedat}{3}
                D_{g_x}\widetilde{\varphi_{x,\varepsilon}}&=0&&\quad\text{in}\quad&& X\backslash B_{2\delta}^{g_x,+}(x)\\
                \widetilde{\varphi_{x,\varepsilon}}&=\varphi_{x,\varepsilon}&&\quad\text{on}\quad&&M\\
                -c_\gamma\lim_{y\rightarrow 0}y^{1-2\gamma}\partial_y\widetilde{\varphi_{x,\varepsilon}}&=c^*_{n,\gamma}\eta_{x,\delta}\delta_{x,\varepsilon}^{\frac{n+2\gamma}{n-2\gamma}}&&\quad\text{on}\quad&&M.
            \end{alignedat}
        \right.
    \]
    Finally, we define the test bubbles as
    \begin{equation}\label{eq: def of test bubble}
        u_{(x,\varepsilon)}:=u_x\varphi_{x,\varepsilon}\quad\textrm{on}\quad M,
    \end{equation}
    where $u_x$ is the conformal factor inducing a conformal normal coordinate system close to $x\in M$.

\section{Proof of \texorpdfstring{\eqref{eq: L 2n / n+2gamma estimate on bubble}}{L2n/(n+2gamma) estimate on bubble}}\label{appendix C}  
    Using $g_x:=u_x^{\frac{4}{n-2\gamma}}g_0$, we compute 
    \begin{equation}\label{eq: Convergence estimate of bubble proof 1}
        \begin{split}
            &\int_M\left|P_\gamma^{g_0}u_{(x,\varepsilon)}-s_\infty u_{(x,\varepsilon)}^{\frac{n+2\gamma}{n-2\gamma}}\right|^{\frac{2n}{n+2\gamma}}\,d\mu_{g_0}\\
            &=\int_M\left|P_\gamma^{g_x}\varphi_{x,\varepsilon}-s_\infty \varphi_{x,\varepsilon}^{\frac{n+2\gamma}{n-2\gamma}}\right|^{\frac{2n}{n+2\gamma}}\,d\mu_{g_x}\\
            &=\lim_{y\rightarrow 0}\int_M\left|c_\gamma y^{1-2\gamma}\partial_y \overline{\varphi_{x,\varepsilon}}+s_\infty \overline{\varphi_{x,\varepsilon}}^{\frac{n+2\gamma}{n-2\gamma}}\right|^{\frac{2n}{n+2\gamma}}\,d\mu_{g_x}\\
            &\leq C\lim_{y\rightarrow 0}\bigg(\int_M \left|y^{1-2\gamma}\partial_y(\overline{\varphi_{x,\varepsilon}}-\widetilde{\varphi_{x,\varepsilon}})\right|^{\frac{2n}{n+2\gamma}}\,d\mu_{g_x}
            \\
            &\hspace{24mm}+\int_M\left|y^{1-2\gamma}\partial_y\widetilde{\varphi_{x,\varepsilon}}^x+s_\infty \widetilde{\varphi_{x,\varepsilon}}^{\frac{n+2\gamma}{n-2\gamma}}\right|^{\frac{2n}{n+2\gamma}}\,d\mu_{g_x}\bigg).
        \end{split}
    \end{equation}
    Note that we have used $\overline{\varphi_{x,\varepsilon}}=\widetilde{\varphi_{x,\varepsilon}}=\varphi_{x,\varepsilon}$ on $M$. 
    
    As long as $n+2\gamma>4$, the first term on the right hand side of \eqref{eq: Convergence estimate of bubble proof 1} is estimated by 
    \begin{equation}\label{eq: Convergence estimate of bubble proof 2}
        \begin{split}
            &\lim_{y\rightarrow 0}\int_M \left|y^{1-2\gamma}\partial_y(\overline{\varphi_{x,\varepsilon}}-\widetilde{\varphi_{x,\varepsilon}})\right|^{\frac{2n}{n+2\gamma}}=o(1),
        \end{split}
    \end{equation}
    by using Lemma \ref{lem: gradient estimate with degenerate right hand side} below, alongside \cite[Lemma 4.1]{mayer2021asymptotics} (note that by using Lemma \ref{lem: gradient estimate with degenerate right hand side} inductively, you can make the singular part integrable in $L^{\frac{2n}{n+2\gamma}}$, for all $n\geq 2$ and $\gamma\in (0,1)$. See Remark \ref{rmk: inductive use of the estimate to remove singularity}.).
    
    On the other hand, using 
    \[
        \lim_{y\rightarrow 0}y^{1-2\gamma}\partial_y\widetilde{\varphi_{x,\varepsilon}}=-s_\infty \delta_{(x,\varepsilon)}^{\frac{n+2\gamma}{n-2\gamma}},
    \]
    (see \cite[p.2582]{mayer2024fractional}) we can estimate the second term: with the choice of $\delta=\varepsilon^{1/k}$, $k\gg 1$, 
    \begin{equation}\label{eq: Convergence estimate of bubble proof 4}
        \begin{split}
            &\lim_{y\rightarrow 0}\int_M\left|y^{1-2\gamma}\partial_y\widetilde{\varphi_{x,\varepsilon}}+s_\infty \widetilde{\varphi_{x,\varepsilon}}^{\frac{n+2\gamma}{n-2\gamma}}\right|^{\frac{2n}{n+2\gamma}}\,d\mu_{g_x}\\
            &\leq \int_M \chi_{B_{2\delta}^{g_x}(x)\backslash B_\delta^{g_x}(x)}\delta_{(x,\varepsilon)}^{\frac{2n}{n-2\gamma}}+\chi_{(B_{2\delta}^{g_x}(x))^c}\varepsilon^n G(\cdot,x)^{\frac{2n}{n-2\gamma}}\,d\mu_{g_x}\\
            &\leq C\left(\int_\delta^{2\delta}r^{n-1}\left(\frac{\varepsilon}{\varepsilon^2+r^2}\right)^n\,dr+\int_\delta^\infty \varepsilon^n r^{n-1}r^{-(n-2\gamma)\frac{2n}{n-2\gamma}}\,dr\right)\\
            &\leq C(\varepsilon/\delta)^n\\
            &=o(1).
        \end{split}
    \end{equation}
    We complete the proof after combining \eqref{eq: Convergence estimate of bubble proof 1}--\eqref{eq: Convergence estimate of bubble proof 2}--\eqref{eq: Convergence estimate of bubble proof 4}.

Finally, we provide the next pointwise gradient estimate from the extension equation with degenerating right hand side, of type \cite[Lemma 4.1, (1)]{mayer2024fractional}.

\begin{lem}\label{lem: gradient estimate with degenerate right hand side}
    Fix $x\in M$, and let $\gamma\in (0,1)$, and $U:B_1^{g_+,+}(x)\rightarrow\mathbb{R}$ be a function satisfying
    \[
        D_{g_x}U=\frac{f}{y^{\gamma}|z|^{1-\gamma}}\quad\textup{in}\quad B_1^{g_+,+}(x),
    \]
    for some $f\in L^\infty(B_2^{g_+,+}(x))$, where $z=(\mathbf{x},y)$ is $g_x$-normal Fermi-coordinate at $x$. Moreover, assume $U\equiv 0$ on $B_1^{g_+,+}(x)$. Then there exists $C=C(n,\gamma,X,g_+)>0$ and $c=c(n,\gamma,X,g_+)>0$ such that
    \begin{equation}\label{eq: pointw+gradientn estimate with rhs}
        y^{-2\gamma}|U(x',y)|\leq C\alpha_{x'},\quad \forall (x',y)\in B_{\tilde{\delta}}^{g_+,+}(x)\cap \{y<c|\mathbf{x}|\},
    \end{equation}
    and
    \begin{equation}
        \lim_{y\rightarrow 0}y^{1-2\gamma}|\partial_y U(x',y)|\leq C\alpha_{x'},\quad\forall x'\in B_{\tilde{\delta}}^{g_x}(x),
    \end{equation}
    where 
    \begin{equation}\label{eq: def of alpha}
        \alpha_{x'} =\max\left\{\frac{\|f\|_{L^\infty(B_1^{g_+,+}(x))}}{d^{1-\gamma^2}_{g_x}(x,x')},\frac{\|f\|_{L^\infty(B_1^{g_+,+}(x))}^{\frac{2}{1+\gamma}}\|U\|_{L^\infty(B_{cd_{g_x}(x,x')/2}^{g_+,+}(x'))}}{d_{g_x}(x,x')^2}\right\}
    \end{equation}
    and
    \[
    \tilde{\delta}\sim_{n,\gamma,X,g_+}\min\{\|f\|_{L^\infty(B_1^{g_+,+}(x))}^{-\frac{1}{2-\gamma-\gamma^2}},\|f\|_{L^\infty(B_1^{g_+,+}(x))}^{-\frac{2(3-\gamma)}{1-\gamma^2}}\|U\|_{L^\infty(B_1^{g_+,+}(x))}^{-\frac{1}{1-\gamma}},1\}.
    \]
\end{lem}

\begin{proof}
    We prove by a barrier argument. Let $\alpha,\beta,\theta>0$ be constants to be chosen later. For each $x'\in B_1^{g_x}(x)$, define a barrier $b:X\rightarrow\mathbb{R}$ by
    \[
        b:=\frac{\alpha}{2}|z'|^2-\beta(y^2-y^{2\gamma})-\frac{\theta y^{1+\gamma}}{|\mathbf{x}|^{1-\gamma}},
    \]
    where abusing the notation, $\mathbf{x}'=\mathbf{x}_{x'}$ and $\mathbf{x}=\mathbf{x}_x$ are $g_x$-normal coordinates centered at $x'$ and $x$, respectively, and $z'=(\mathbf{x}',y)$.
    Then
    \[
        \left\{
            \begin{alignedat}{2}
                D_{g_x}|z'|^2&=(-y^{1-2\gamma}(\Delta_{g_x}+(1-2\gamma)y^{-1}\partial_y|z'|^2)+E_{g_x})|z'|^2\\
                D_{g_x}y^2&=-2(2-2\gamma)y^{1-2\gamma}+E_{g_x}y^2\\
                D_{g_x}y^{2\gamma}&=E_gy^{2\gamma},\\
                D_{g_x}\frac{y^{1+\gamma}}{|\mathbf{x}|^{1-\gamma}}&=-\frac{1-\gamma^2}{y^\gamma|\mathbf{x}|^{1-\gamma}}-y^{2-\gamma}\Delta_{g_x}|\mathbf{x}|^{\gamma-1}+E_{g_x}\frac{y^{1+\gamma}}{|\mathbf{x}|^{1-\gamma}}.
            \end{alignedat}
        \right.
    \]
    Hence
    \begin{equation}\label{eq: expansion of Db}
        \begin{split}
            &D_{g_x}b\\
            &=\left(\frac{\alpha}{2}(-\Delta_{g_x}-(1-2\gamma)y^{-1}\partial_y)|z'|^2+2(2-2\gamma)\beta \right)y^{1-2\gamma}+\theta\left(\frac{1-\gamma^2}{y^\gamma|\mathbf{x}|^{1-\gamma}}+y^{2-\gamma}\Delta_{g_x}|\mathbf{x}|^{\gamma-1}\right)\\
            &\hspace{4mm}+E_{g_x}\left(\frac{\alpha}{2}|z'|^2-\beta y^2+\beta y^{2\gamma}-\theta\frac{y^{1+\gamma}}{|\mathbf{x}|^{1-\gamma}}\right)\\
            &\geq \theta\left(\frac{1-\gamma^2}{y^\gamma|\mathbf{x}|^{1-\gamma}}+y^{2-\gamma}\Delta_{g_x}|\mathbf{x}|^{\gamma-1}\right)+E_{g_x}\left(\frac{\alpha}{2}|z'|^2-\beta y^2+\beta y^{2\gamma}-\theta\frac{y^{1+\gamma}}{|\mathbf{x}|^{1-\gamma}}\right),
        \end{split}
    \end{equation}
    provided we take 
    \[
    \beta\geq \frac{\sup_{z'\in B_1^{g_+,+}(x)}\|(\Delta_{g_x}+(1-2\gamma)y^{-1}\partial_y)|z'|^2\|_{L^\infty (B_1^{g_+,+}(x))}}{8(1-\gamma)}\alpha.
    \]
    The coefficient on the right hand side is finite (because $M$ is locally flat), and only depends on $n,\gamma,X,g_+$.

    Additionally, there exists $c=c(n,\gamma,X,g_+)>0$, such that if $y<c|\mathbf{x}|$, then
    \begin{equation}\label{eq: lower cone condition}
        y^{2-\gamma}\Delta_{g_x}|\mathbf{x}|^{\gamma-1}<\frac{1-\gamma^2}{2y^\gamma|\mathbf{x}|^{1-\gamma}}.
    \end{equation}
    Moreover, from \eqref{eq: Egx approximation} and \eqref{eq: decying order in Egx from zero mean curvature},
    \[
        |E_{g_x}|\lesssim_{n,\gamma,X,g_+} y^{-\gamma},
    \]
    and thus
    \begin{equation}\label{eq: Egx term estimate}
        E_{g_x}\left(\frac{\alpha}{2}|\mathbf{x}|^2-\beta y^2+\beta y^{2\gamma}-\theta\frac{y^{1+\gamma}}{|\mathbf{x}|^{1-\gamma}}\right)\geq -\frac{\theta(1-\gamma^2)}{4y^\gamma |\mathbf{x}|^{1-\gamma}},\quad\textrm{in}\quad B_{\tilde{\delta}}^{g_+,+}(x),
    \end{equation}
    for some small 
    \[
    \tilde{\delta}\lesssim_{n,\gamma}\min\{\alpha^{-\frac{1}{3-\gamma}},\beta^{-\frac{1}{1+\gamma}}\}\sim_{n,\gamma,X,g_+}\min\{\alpha^{-\frac{1}{3-\gamma}},\alpha^{-\frac{1}{1+\gamma}}\}.
    \]
    (note that we have chosen $\beta\sim_{n,\gamma,X,g_+}\alpha$). Combine \eqref{eq: expansion of Db}, \eqref{eq: lower cone condition}, and \eqref{eq: Egx term estimate}, we conclude
    \begin{equation}\label{eq: supersolution condition, D bound}
        D_{g_x}b>\frac{\theta (1-\gamma^2)}{4y^\gamma|\mathbf{x}|^{1-\gamma}}\geq \frac{\theta (1-\gamma^2)}{4y^\gamma|z|^{1-\gamma}}>D_{g_x}U,\quad\text{in}\quad B_{\tilde{\delta}}^{g_+,+}(x)\cap\{y<c|\mathbf{x}|\},
    \end{equation}
    by taking $\theta\sim_{n,\gamma} \|f\|_{L^\infty(B_1^{g_+,+}(x))}$. 
    
    Now, take $c$ further small so that $c<\min\left\{\left(\frac{\alpha}{4\theta}\right)^{\frac{1}{1+\gamma}},1\right\}$, and then choose 
    \[
    \alpha\gtrsim_{\gamma}\max\left\{\frac{\theta}{d_{g_x}^{1-\gamma^2}(x,x')},\frac{\theta^{\frac{2}{1+\gamma}}\|U\|_{L^\infty(B_{cd_{g_x}(x,x')/2}^{g_+,+}(x'))}}{d_{g_x}^2(x,x')}\right\}.
    \]
    One can compute that with such choice, we can take balls in inside of conical region, i.e. $B_{cd_{g_x}(x,x')/2}^{g_+,+}(x')\subset \{y<c|\mathbf{x}|\}$, and 
    \begin{equation}\label{eq: Supersolution condition, Linfty bound}
            b(\mathbf{x},y)=\frac{\alpha}{2}|z'|^2-\beta(y^2-y^{2\gamma})-\frac{\theta y^{1+\gamma}}{|\mathbf{x}|^{1-\gamma}}>\|U\|_{L^\infty(B_{cd_{g_x}(x,x')/2}^{g_x}(x))},\quad\text{on}\quad \partial B_{cd_{g_x}(x,x')/2}^{g_+,+}(x').
    \end{equation}
    Indeed, with such choice of $\alpha$, we have
    \[
        \alpha\left(\alpha^{\frac{2}{1+\gamma}}-\frac{C_\gamma\theta^{\frac{2}{1+\gamma}}}{d_{g_x}^{2-2\gamma}(x,x')}\right)>\frac{C_\gamma \theta^{\frac{2}{1+\gamma}}\|U\|_{L^\infty (B_{cd_{g_x}(x,x')/2}^{g_x}(x))}}{d_{g_x}^2(x,x')},\quad\text{on}\quad \partial B_{cd_{g_x}(x,x')/2}^{g_+,+}(x')
    \]
    which implies that 
    \[
        \alpha >\left(2\|U\|_{L^\infty (B_{cd_{g_x}(x,x')/2}^{g_x}(x))}+\frac{2\theta\hat{\delta}^{1+\gamma}}{(d_{g_x}(x,x')-\hat{\delta})^{1-\gamma}}\right)\frac{1}{\hat{\delta}^2},\quad\text{on}\quad \partial B_{cd_{g_x}(x,x')/2}^{g_+,+}(x')
    \]
    where $\hat{\delta}:=\frac{c}{2}d_{g_x}(x,x')$, with $c\lesssim_{n,\gamma,X,g_+}\left(\frac{\alpha}{\theta}\right)^{\frac{1}{1+\gamma}}$, as before. From here \eqref{eq: Supersolution condition, Linfty bound} follows.
    
    Moreover, from the above choice of $\tilde{\delta}$, by taking
    \[ 
        d_{g_x}(x,x')\lesssim_{n,\gamma,X,g_+} \min\{\theta^{-\frac{1}{2-\gamma-\gamma^2}},\theta^{-\frac{2(3-\gamma)}{1-\gamma^2}}\|U\|_{L^\infty(B_1^{g_+,+}(x))}^{-\frac{1}{1-\gamma}}\},
    \]
    we can have
    \[
        \tilde{\delta}>2d_{g_x}(x,x').
    \]
    That is, $B_{cd_{g_x}(x,x')/2}^{g_+,+}(x')\subset B_{\tilde{\delta}}^{g_+,+}(x)$. Therefore, we choose
    \[
        \tilde{\delta}\sim_{n,\gamma,X,g_+}\min\{\|f\|_{L^\infty(B_1^{g_+,+}(x))}^{-\frac{1}{2-\gamma-\gamma^2}},\|f\|_{L^\infty(B_1^{g_+,+}(x))}^{-\frac{2(3-\gamma)}{1-\gamma^2}}\|U\|_{L^\infty(B_1^{g_+,+}(x))}^{-\frac{1}{1-\gamma}},1\}.
    \]
    
    Then, from \eqref{eq: supersolution condition, D bound} and \eqref{eq: Supersolution condition, Linfty bound}, it follows that $b$ is a supersolution to $U$, in $B_{cd_{g_x}(x,x')/2}^{g_+,+}(x')$. Moreover, the construction of the subsolution follows similarly. Thus by the maximum principle, for any $(\mathbf{x}',y)\in B_{cd_{g_x}(x,x')/2}^{g_+,+}(x')$,
    \[
        |U(\mathbf{x}',y)|\leq b(\mathbf{x}',y)\leq \frac{\alpha}{4}|\mathbf{x}'|^2+\beta y^{2\gamma}
    \]

    Now, apply the same argument for arbitrary $x'\in B_{\tilde{\delta}/2}^{g_+,+}(x)$, we conclude that
    \[
        |U(\mathbf{x},y)|\leq \beta y^{2\gamma}\sim_{n,\gamma,X,g_+}\alpha y^{2\gamma},\quad\forall (\mathbf{x},y)\in \{y<c|\mathbf{x}|\}\cap B_{\tilde{\delta}/2}^{g_+,+}(x).
    \]
    By the standard interior elliptic estimate and letting the limit $y\rightarrow 0$, we also obtain the gradient estimate. This completes the proof.
\end{proof}

\begin{remark}\label{rmk: inductive use of the estimate to remove singularity}
    The singular part in $\alpha$ is in $L_\text{loc}^{\frac{2n}{n+2\gamma}}(M)$ if $n+2\gamma>4$. However, one can apply the above lemma inductively to remove this restriction, for any $n\geq 2$ and $\gamma\in (0,1)$. To be more precise, for any $x'\in M$, define $D_0(x'):=B_{cd_{g_x}(x,x')/2}^{g_+,+}(x')$, and a set consist of balls centered in a points in $D_0(x')$ by
    \[
        D_1(x'):=\bigcup_{x_1\in D_0(x')}B^{g_+,+}_{cd_{g_x}(x,x_1)/2}(x_1).
    \]
    Then $D_1(x')$ is still contained in the cone $\{y<c|\mathbf{x}|\}$, and thus we can apply the pointwise estimate \eqref{eq: pointw+gradientn estimate with rhs} (if the maximum in the definition of $\alpha$ was achieved by the second factor) again, to have
    \begin{align*}
        &\|U\|_{L^\infty(D_0(x'))}\\
        &\lesssim_{n,\gamma,X,g_+}\left(\sup_{x_1\in D_0(x')}\alpha_{x_1} \right)(cd_{g_x}(x,x'))^{2\gamma}\\
        &\lesssim_{n,\gamma,X,g_+,\|f\|_{L^\infty(B_1^{g_+,+})(x)}}\frac{\|U\|_{L^\infty\left(D_1(x')\right)}}{\left(\inf_{x_1\in D_0(x')}d_{g_x}^2(x,x_1)\right)d_{g_x}^{-2\gamma}(x,x')}\\
        &\lesssim_{n,\gamma,X,g_+}\frac{\|U\|_{L^\infty(B_1^{g_+,+})}}{d_{g_x}^{2-2\gamma}(x,x')}.
    \end{align*}
    In the above, we are estimating the pointwise value of $U$ in the respective ball $D_1(x')$, by taking a collection of balls centered in each points in the first ball $D_0(x')$, which are still in the conical region $\{y<c|\mathbf{x}|\}$. As a result, we gained the integrability from $d_{g_x}^{-2}$ in \eqref{eq: def of alpha} to $d^{2\gamma-2}_{g_x}$, in sacrifice of the coefficient. One can repeat this procedure finitely many times by defining $D_k(x')$ in the inductive manner (and all $D_k(x')$ are still in the conical region $\{y<c|\mathbf{x}|\}$) to obtain $L^{\frac{2n}{n+2\gamma}}$ integrability for the right hand side of \eqref{eq: pointw+gradientn estimate with rhs}, for any $n\geq 2$ and $\gamma\in (0,1)$. Note that we cannot completely remove the singular denominator and obtain $L^\infty$ estimate in this method, because $\frac{\|f\|_{L^\infty (B_1^{g_+,+}(x))}}{d_{g_x}^{1-\gamma^2}(x,x')}$ in \eqref{eq: def of alpha} does not improve in this inductive step.
\end{remark}

%% file: Main.bbl
\begin{thebibliography}{27}
\providecommand{\natexlab}[1]{#1}
\providecommand{\url}[1]{\texttt{#1}}
\expandafter\ifx\csname urlstyle\endcsname\relax
  \providecommand{\doi}[1]{doi: #1}\else
  \providecommand{\doi}{doi: \begingroup \urlstyle{rm}\Url}\fi

\bibitem[Aubin(1998)]{aubin1998some}
T.~Aubin.
\newblock \emph{Some {N}onlinear {P}roblems in {R}iemannian {G}eometry}.
\newblock Springer Science \& Business Media, 1998.

\bibitem[Brendle(2005)]{Brendle}
S.~Brendle.
\newblock Convergence of the {Y}amabe {F}low for {A}rbitrary {I}nitial {E}nergy.
\newblock \emph{Journal of Differential Geometry}, 69\penalty0 (2):\penalty0 217--278, 2005.

\bibitem[Brendle(2007)]{brendle2007convergence}
S.~Brendle.
\newblock Convergence of {T}he {Y}amabe {F}low in {D}imension 6 and {H}igher.
\newblock \emph{Inventiones mathematicae}, 170\penalty0 (3), 2007.

\bibitem[Caffarelli and Silvestre(2007)]{Caffarelli-Silvestre}
L.~Caffarelli and L.~Silvestre.
\newblock An {E}xtension {P}roblem {R}elated to the {F}ractional {L}aplacian.
\newblock \emph{Comm. Partial Differential Equations}, 32\penalty0 (7-9):\penalty0 1245--1260, 2007.

\bibitem[Chan et~al.(2020)Chan, Sire, and Sun]{ChanSireSun}
H.~Chan, Y.~Sire, and L.~Sun.
\newblock Convergence of the {F}ractional {Y}amabe {F}low for a {C}lass of {I}nitial {D}ata.
\newblock \emph{Annali Scienze}, XXI\penalty0 (21):\penalty0 1703--1740, 2020.

\bibitem[Chang and del Mar~Gonzalez(2011)]{chang2011fractional}
S.-Y.~A. Chang and M.~del Mar~Gonzalez.
\newblock Fractional {L}aplacian in {C}onformal {G}eometry.
\newblock \emph{Advances in Mathematics}, 226\penalty0 (2):\penalty0 1410--1432, 2011.

\bibitem[Chow(1992)]{chow1992yamabe}
B.~Chow.
\newblock The yamabe {F}low on {L}ocally {C}onformally {F}lat {M}anifolds with {P}ositive {R}icci {C}urvature.
\newblock \emph{Communications on pure and applied mathematics}, 45\penalty0 (8):\penalty0 1003--1014, 1992.

\bibitem[Daskalopoulos et~al.(2017)Daskalopoulos, Sire, and V{\'a}zquez]{daskalopoulos2017weak}
P.~Daskalopoulos, Y.~Sire, and J.-L. V{\'a}zquez.
\newblock Weak and {S}mooth {S}olutions for a {F}ractional {Y}amabe {F}low: the {C}ase of {G}eneral {C}ompact and {L}ocally {C}onformally {F}lat {M}anifolds.
\newblock \emph{Communications in Partial Differential Equations}, 42\penalty0 (9):\penalty0 1481--1496, 2017.

\bibitem[del Mar~Gonz{\'a}lez et~al.(2012)del Mar~Gonz{\'a}lez, Mazzeo, and Sire]{del2012singular}
M.~del Mar~Gonz{\'a}lez, R.~Mazzeo, and Y.~Sire.
\newblock Singular {S}olutions of {F}ractional {O}rder {C}onformal {L}aplacians.
\newblock \emph{Journal of Geometric Analysis}, 22\penalty0 (3):\penalty0 845--863, 2012.

\bibitem[Di~Nezza et~al.(2012)Di~Nezza, Palatucci, and Valdinoci]{DPV}
E.~Di~Nezza, G.~Palatucci, and E.~Valdinoci.
\newblock Hitchhiker's {G}uide to the {F}ractional {S}obolev {S}paces.
\newblock \emph{Bulletin des sciences math{\'e}matiques}, 136\penalty0 (5):\penalty0 521--573, 2012.

\bibitem[Fang and Gonz{\'a}lez~Nogueras(2015)]{FangG}
Y.~Fang and M.~d.~M. Gonz{\'a}lez~Nogueras.
\newblock Asymptotic {B}ehavior of {P}alais--{S}male {S}equences {A}ssociated with {F}ractional {Y}amabe-type {E}quations.
\newblock \emph{Pacific Journal of Mathematics}, 278\penalty0 (2):\penalty0 369--405, 2015.

\bibitem[Gonz{\'a}lez~Nogueras and Qing(2013{\natexlab{a}})]{DQ}
M.~d.~M. Gonz{\'a}lez~Nogueras and J.~Qing.
\newblock Fractional {C}onformal {L}aplacians and {F}ractional {Y}amabe {P}roblems.
\newblock \emph{Analysis \& PDE}, 6\penalty0 (7):\penalty0 1535--1576, 2013{\natexlab{a}}.

\bibitem[Gonz{\'a}lez~Nogueras and Qing(2013{\natexlab{b}})]{gonzalez2013fractional}
M.~d.~M. Gonz{\'a}lez~Nogueras and J.~Qing.
\newblock Fractional {C}onformal {L}aplacians and {F}ractional {Y}amabe {P}roblems.
\newblock \emph{Analysis \& PDE}, 6\penalty0 (7):\penalty0 1535--1576, 2013{\natexlab{b}}.

\bibitem[Graham and Zworski(2003)]{graham2003scattering}
C.~R. Graham and M.~Zworski.
\newblock Scattering {M}atrix in {C}onformal {G}eometry.
\newblock \emph{Invent. Math.}, 152\penalty0 (1):\penalty0 89--118, 2003.

\bibitem[Graham et~al.(1992)Graham, Jenne, Mason, and Sparling]{graham1992conformally}
C.~R. Graham, R.~Jenne, L.~J. Mason, and G.~A. Sparling.
\newblock Conformally {I}nvariant {P}owers of the {L}aplacian, {I}: {E}xistence.
\newblock \emph{Journal of the London Mathematical Society}, 2\penalty0 (3):\penalty0 557--565, 1992.

\bibitem[Hamilton(1988)]{hamilton1988ricci}
R.~S. Hamilton.
\newblock The {R}icci {F}low on {S}urfaces, {M}athematics and {G}eneral {R}elativity.
\newblock \emph{Contemp. Math.}, 71:\penalty0 237--261, 1988.

\bibitem[Jin and Xiong(2014)]{jin2014fractional}
T.~Jin and J.~Xiong.
\newblock A {F}ractional {Y}amabe {F}low and {S}ome {A}pplications.
\newblock \emph{Journal f{\"u}r die reine und angewandte Mathematik (Crelles Journal)}, 2014\penalty0 (696):\penalty0 187--223, 2014.

\bibitem[Kim et~al.(2017)Kim, Musso, and Wei]{KimMussoWei}
S.~Kim, M.~Musso, and J.~Wei.
\newblock Existence {T}heorems of the {F}ractional {Y}amabe {P}roblem.
\newblock \emph{Analysis \& PDE}, 11\penalty0 (1):\penalty0 75--113, 2017.

\bibitem[Mayer and Ndiaye(2022)]{mayer2021asymptotics}
M.~Mayer and C.~B. Ndiaye.
\newblock Asymptotics of the {P}oisson kernel and {G}reen's functions of the fractional conformal {L}aplacian.
\newblock \emph{Discrete Contin. Dyn. Syst.}, 42\penalty0 (10):\penalty0 5037--5062, 2022.

\bibitem[Mayer and Ndiaye(2024)]{mayer2024fractional}
M.~Mayer and C.~B. Ndiaye.
\newblock Fractional {Y}amabe {P}roblem on {L}ocally {F}lat {C}onformal {I}nfinities of {P}oincare-{E}instein {M}anifolds.
\newblock \emph{International Mathematics Research Notices}, 2024\penalty0 (3):\penalty0 2561--2621, 2024.

\bibitem[Palatucci and Pisante(2015)]{PalatucciPisante}
G.~Palatucci and A.~Pisante.
\newblock A {G}lobal {C}ompactness {T}ype {R}esult for {P}alais--{S}male {S}equences in {F}ractional {S}obolev {S}paces.
\newblock \emph{Nonlinear Analysis: Theory, Methods \& Applications}, 117:\penalty0 1--7, 2015.

\bibitem[Qing and Raske(2006)]{qing2006compactness}
J.~Qing and D.~Raske.
\newblock Compactness for {C}onformal {M}etrics with {C}onstant {Q} {C}urvature on {L}ocally {C}onformally {F}lat {M}anifolds.
\newblock \emph{Calculus of Variations and Partial Differential Equations}, 26\penalty0 (3):\penalty0 343--356, 2006.

\bibitem[Schoen(1984)]{schoen1984conformal}
R.~Schoen.
\newblock Conformal {D}eformation of a {R}iemannian {M}etric to {C}onstant {S}calar {C}urvature.
\newblock \emph{Journal of Differential Geometry}, 20\penalty0 (2):\penalty0 479--495, 1984.

\bibitem[Schwetlick and Struwe(2003)]{schwetlick2003convergence}
H.~Schwetlick and M.~Struwe.
\newblock Convergence of the {Y}amabe {F}low for {L}arge {E}nergies.
\newblock \emph{Journal für die reine und angewandte mathematik}, 2003.

\bibitem[Trudinger(1968)]{trudinger1968remarks}
N.~S. Trudinger.
\newblock Remarks {C}oncerning the {C}onformal {D}eformation of {R}iemannian {S}tructures on {C}ompact {M}anifolds.
\newblock \emph{Annali della Scuola Normale Superiore di Pisa-Scienze Fisiche e Matematiche}, 22\penalty0 (2):\penalty0 265--274, 1968.

\bibitem[Yamabe(1960)]{yamabe1960deformation}
H.~Yamabe.
\newblock On a {D}eformation of {R}iemannian {S}tructures on {C}ompact {M}anifolds.
\newblock \emph{Osaka Math. J.}, 12:\penalty0 21--37, 1960.

\bibitem[Ye(1994)]{ye1994global}
R.~Ye.
\newblock Global {E}xistence and {C}onvergence of {Y}amabe {F}low.
\newblock \emph{Journal of Differential Geometry}, 39\penalty0 (1):\penalty0 35--50, 1994.

\end{thebibliography}
